\documentclass[a4paper,12pt]{article}
\usepackage[dvips]{graphics,graphicx,epsfig}
\def\rien{\rule{0pt}{0pt}}
\usepackage{amssymb}
\usepackage{amsmath}
\usepackage[usenames]{color}
\usepackage{bussproofs}
\usepackage{enumerate}
\usepackage{amsthm}

\def\Nset{\mathrm{I}\!\mathrm{N}}
\def\Zset{\mathbb{Z}}

\def\Rset{\mathrm{I}\!\mathrm{R}}
\def\RsetAlt{\mathcal{R}}

\def\thegroupAlone{G}
\def\projFromHahn#1#2{\left(#1\right)}
\def\theHahnSeriesField#1#2{#1\left(\left(#2\right)\right)}
\def\thegroup#1#2{\thegroupAlone_{#1,#2}}
\def\theHahnSupport#1{\sigma\left(#1\right)}

\def\directlim{{\displaystyle\lim_{\longrightarrow}\;}}

\def\equivrel{\approx}
\def\everFalse{\emptyset}
\def\everTrue{\Omega}
\def\themapping#1{\mu_{#1}}
\def\thejointmapping#1{(\mu\!\times\!\mu)_{#1}}
\def\theCondMapping#1{\rho_{#1}}

\def\theid#1{{{\mathrm{id}}_{{#1}}}}
\def\thetransp#1{{T_{{#1}}}}

\def\preeverFalse{{\overline{\everFalse}}}
\def\thepartof#1{{\sigma\!#1}}

\def\thecropof#1#2{{#1[#2]}}
\def\thesetofProb#1#2{\mathbb{P}_{#1}\left(#2\right)}
\def\thesetofBayesProb#1#2{\mathbf{I\!P}_{#1}\left(#2\right)}
\def\thecardof#1{{\mathrm{card}(#1)}}

\def\thepropBool{\mbox{$\mathbf{Bool}$}}
\def\thepropDef{\mbox{$\mathbf{Def}$}}
\def\thepropInf{\mbox{$\mathbf{Inf}$}}
\def\thepropInd{\mbox{$\mathbf{Ind}$}}

\def\thesetofFiniteSubboolean#1{{\mathbb{B}\left(#1\right)}}
\def\thebayesextent#1{\overline{#1}}

\newtheorem{theorem}{Theorem}
\newtheorem{property}{Property}
\newtheorem{lemma}[property]{Lemma}
\newtheorem{corollary}[property]{Corollary}
\newtheorem{definition}[property]{Definition}

\newtheorem{example}[property]{Example}
\newtheorem{notation}[property]{Notation}

\def\theAxInfIIcond{\mbox{$\mathbf{Cd\ }$}}
\def\theAxK{\mbox{$\mathbf{\ K\ }$}}
\def\theAxCondIIinf{\mbox{$\mathbf{Cd^{-1}}$}}
\def\theAxNeg{\mbox{$\mathbf{Neg}$}}
\def\theAxInd{\mbox{$\mathbf{Ind}$}}
\def\metaP{\mbox{$\mathbf{\mu P}$}} 
\def\metaC{\mbox{$\mathbf{\mu C}$}} 
\def\metaW{\mbox{$\mathbf{\mu W}$}} 

\def\thesetofProp{\mathcal{L}_{C}}
\def\thesetofCondProp{\mathcal{L}_{[\,]}}
\def\thesetofDBLProp{\mathcal{L}}

\author{Fr\'ed\'eric Dambreville}
\title{Extension of Boolean algebra by a Bayesian operator; application to the definition of a Deterministic Bayesian Logic}
\begin{document}
%\dominitoc
\maketitle

\begin{center}
%\Axiom$\fCenter\hsstart [X]X\rightarrow X\hsstop$
%\Axiom$\fCenter \hsstart X\rightarrow[X]X\hsstop$
%\doubleLine
%\LeftLabel{}\RightLabel{}
%\BinaryInf$\fCenter \hsstart X\hssep\neg X\hsstop$
%\DisplayProof
$[x]x=x\Rightarrow x\in\{\everFalse,\everTrue\}$
 \ --- \ 
\emph{Free of itself, it is the all or the none.}
\end{center}
\begin{abstract}
This work contributes to the domains of Boolean algebra and of Bayesian probability, by proposing an algebraic extension of Boolean algebras, which implements an operator for the Bayesian conditional inference and is closed under this operator.
It is known since the work of Lewis (Lewis' triviality) that it is not possible to construct such conditional operator within the space of events.
Nevertheless, this work proposes an answer which complements Lewis' triviality, by the construction of a conditional operator outside the space of events, thus resulting in an algebraic extension.
In particular, it is proved that any probability defined on a Boolean algebra may be extended to its algebraic extension in compliance with the multiplicative definition of the conditional probability.
In the last part of this paper, a new \emph{bivalent} logic is introduced on the basis of this algebraic extension, and basic properties are derived.
\end{abstract}

\section{Introduction}
Many implementations of practical problems make apparent the logical nature of conditional probabilities, which are kinds of inference operators.
This fact typically led to the development of various Bayesian approach for manipulating uncertain logical information (Bayesian networks, Bayesian logic,\dots)
The interpretation of conditional probabilities as logical inferences naturally introduced the question of the definition of conditionals directly at the propositional level:
\emph{is it possible to define conditional probabilities of events as probabilities of conditional events?}
A negative answer to this question was given by Lewis' triviality~\cite{lewis}\,,
which implies that it is not possible to define a conditional operator within the space of unconditional events -- \emph{c.f.} property~\ref{prop:lewis:1}.
However, Lewis' triviality does not forbid the construction of conditional operators by means of an extension of the space of event.
In accordance with this observation, the purpose of this paper is to prove the following main theorem, which asserts the existence of such extension when working on Boolean algebras (this result has not been generalized to measurable spaces at this time).
\begin{theorem}[Bayesian extension of Boolean algebra]\label{main:theorem:1}
Let $(B_{oole},\cap,\cup,\sim,\everFalse,\everTrue)$ be a Boolean algebra.
%, which admits a strictely positive (finitely additive) probability measure.
Then there is a septuple $(B_{ayes},\cap,\cup,\sim,\everFalse,\everTrue,[\;])$ such that:
\begin{itemize}
\item $B_{ayes}$\,, considered as $(B_{ayes},\cap,\cup,\sim,\everFalse,\everTrue)$\,, is a Boolean algebra,
%\item $B_{ayes}$ is countable\,,
\item There is an injective Boolean morphism $\themapping{}:B_{oole}\rightarrow B_{ayes}$\,,
\item The operator $[\;]$ is such that:
\begin{itemize}
\item $z\mapsto [x]z$ is a Boolean automorphism of $B_{ayes}$\,,
\item $x\subset y$ implies $[x]y=\everTrue$ or $x=\everFalse$\,,
\item $x\cap[x]y=x\cap y$\,,
\item $[\sim x][x]y=[x][x]y=[x]y$\,,
\end{itemize}
for all $x,y\in B_{ayes}$\,.
\item Given any probability distribution $P_{oole}$ defined on $B_{oole}$\,, there is a probability distribution $P_{ayes}$ defined on $B_{ayes}$ such that
%:
%\begin{itemize}
%\item 
$P_{ayes}\circ\themapping{} = P_{oole}$
%\,,
%\item 
and:
\begin{equation}\label{equn:indep:rel:1}
P_{ayes}(x\cap y) = P_{ayes}([x]y)P_{ayes}(x)
\mbox{ \ for all }x,y\in B_{ayes}\;.
\end{equation}
%\end{itemize}
\end{itemize}
\end{theorem}
%
%Typically, this result holds for all countable Boolean algebras and for all free Boolean algebra.
%
This result, extending the structure of Boolean space, provides an example of algebraic construction of a Bayesian space, which is closed under the conditional operator $[\;]$\,.
It is noticed that there are in the domain of \emph{conditional event algebra} interesting examples of algebraic construction of \emph{external} conditional proposition -- for example \cite{nguyen1999}.
However, nested conditional propositions, obtained from known closures of such algebras under the conditional operator, are not compliant with the conditional relation~(\ref{equn:indep:rel:1}).
The difficulty of an algebraic construction of the conditional has been pointed by Lewis' triviality~\cite{lewis}. This result is recalled now, in the general framework of measurable spaces. 
\paragraph{Lewis' triviality.}
\begin{property}\label{prop:lewis:1}
Let $(\Omega, \mathcal{F})$ be a measurable space.
Let be defined $[\;]:\mathcal{F}\times \mathcal{F}\rightarrow \mathcal{F}$ such that $P\bigl([x]y\bigr)P(x)=P(x\cap y)$ for all $x,y\in \mathcal{F}$ and $P$ a probability distribution on $\mathcal{F}$\,.
\\[3pt]
Let $x,y\in\mathcal{F}$ and a probability distribution $P$ such that $P(x\cap y)>0$ and $P(\sim x\cap y)>0$\,.
Then $P(x\cap y)=P(x)P(y)$\,.
\end{property}
\begin{proof}
Define $P_{x}(y)=P\bigl([x]y\bigr)$ and $P_{\sim x}(y)=P\bigl([\sim x]y\bigr)$\,.
Then:
$$
P_{x}\bigl([y]x\bigr)=\frac{P_{x}(x\cap y)}{P_{x}(y)}=
\frac{\frac{P(x\cap y)}{P(x)}}{\frac{P(x\cap y)}{P(x)}}=
1
\mbox{ and }
P_{\sim x}\bigl([y]x\bigr)=\frac{P_{\sim x}(x\cap y)}{P_{\sim x}(y)}=
\frac{\frac{P(\sim x\cap x\cap y)}{P(\sim x)}}{\frac{P(\sim x \cap y)}{P(\sim x)}}=
0\;.
$$
Then:
\[
\frac{P(x\cap y)}{P(y)}=P\bigl([y]x\bigr)=
P(x)P_x\big([y]x\big)+P(\sim x)P_{\sim x}\big([y]x\big)
=P(x)+0=P(x)\;.
\qedhere\]
\end{proof}
In particular, the existence of $[\;]$ implies that it is impossible to have $x\subset y$ such that $0<P(x)<P(y)<1$\,.
This result is irrelevant.
\\[3pt]
As explained previously, the triviality is based on the hypothesis that $[x]y\in \mathcal{F}$\,, which makes possible the previous computation of $P_{x}\bigl([y]x\bigr)$ and $P_{\sim x}\bigl([y]x\bigr)$\,.
In theorem~\ref{main:theorem:1} however, the conditional operator is constructed outside of the measured Boolean space, and the triviality is thus avoided.
\\[5pt]
This presentation consists of two sections.
The main section~\ref{sect:proof} establishes a proof of theorem~\ref{main:theorem:1}.
It is based on the construction of an algebraic extension of Boolean algebras, by introducing a conditional operator compatible with a probability extension.
Based on this work, a notion of Bayesian algebra is introduced.
As an application, section~\ref{sect:logic} deals with the logical interpretation of these Bayesian algebras.
%\\[3pt]
Section~\ref{sect:conclusion} concludes.
\section{Proof of the main theorem}\label{sect:proof}
A proof of main theorem~\ref{main:theorem:1} is derived throughout this section.
It is first recalled in section~\ref{sect:dirlim} and section~\ref{sect:bool} some useful tools, which will be instrumental for the construction of a Bayesian algebra.
The construction of the algebra and the problem of probability extension on this algebra are done in section~\ref{sect:bayes:model}.
%, and section~\ref{sect:proba:extension} addresses the problem of probability extension on this algebra.
At last, section~\ref{sect:main:theorem} compiles these results, thus achieving the proof of the main theorem.

\paragraph{Conventions.}
\begin{itemize}
%\item For the sack of simplicity, a structured set $(E,\ast_1,\cdots,\ast_n)$\, -- $E$ is a set, $\ast_i$ is an operator, a relation, a constant, $\cdots$ -- will be simply denoted $E$,
\item Notations $1:n$\,, $x_{1:n}$ and $i=1:n$ stand respectively for the sequences $1,\cdots, n$\,, $x_1,\cdots,x_n$ and the relation $i\in\{1,\cdots, n\}$\,.%,
%\item Notations $x,y,z \in E$ stands for $x\in E\,,$ $y\in E$ and $z\in E$\,. 
\end{itemize}

\subsection{Direct limit}\label{sect:dirlim}
Hereinafter, direct limits will be quite useful tools for constructing Bayesian extensions of Boolean algebras.
Thorough references on direct limits may be found in~\cite{bourbaki:algebra} and in~\cite{maclane:1998}.
\subsubsection{Basic notions}
A basic introduction to directs limits is done now.
These known results are presented without proofs.
\begin{definition}[directed set]
A (partially) ordered set $(I,\le)$ is a directed set if there is $k\in I$ such that $k\ge i$ and $k\ge j$ for all $i,j\in I$\,.
\end{definition}
\begin{definition}[direct system]
Let $(I,\le)$ be a directed set.
Let $(E_i)_{i\in I}$ be a sequence of structured set of same nature, and let $\themapping{i,j}:E_i\longrightarrow E_j$ be a morphism defined for all $i\le j$ with the properties $\themapping{i,i}=\theid{E_i}$ and $\themapping{j,k}\circ \themapping{i,j}=\themapping{i,k}$ for all $i\le j\le k$\,.
The pair $(E_i,\themapping{i,j})_{\substack{i,j\in I\\i\le j}}$ is called a direct system.
\end{definition}
\begin{property}
Let $(E_i,\themapping{i,j})_{\substack{i,j\in I\\i\le j}}$ be a direct system.
Let $\equivrel$ be a relation defined on $\bigsqcup_{i\in I}E_i$, the disjoint union of the set $E_i$\,, by:
$$
x_i\equivrel y_j \mbox{ if and only if there is }k\ge i,j\mbox{ such that }\themapping{i,k}(x_i)=\themapping{j,k}(x_j)\;.
$$
The relation $\equivrel$ is an equivalence relation.
Moreover, this relation is compatible with the structures of $(E_i)_{i\in I}$\,.
\end{property}
\begin{definition}[direct limit]
The direct limit $\directlim E_i$ of a direct system $(E_i,\themapping{i,j})_{\substack{i,j\in I\\i\le j}}$ is defined by:
$$
\directlim E_i = \left.\bigsqcup_{i\in I}E_i\right/\equivrel
\;,
$$
the set of classes of equivalence of $\bigsqcup_{i\in I}E_i$\,.
It is defined the canonical mapping $\themapping{i}$\,, which maps the elements of $E_i$ to their equivalence class:
$$
\themapping{i}(x_i)= \left.x_i\right/\equivrel\mbox{ \ for all }i\in I\mbox{ and }x_i\in E_i\;.
$$
\end{definition}

\begin{property}[structure inheritance]
The direct limit $\directlim E_i$ inherits their structure from $(E_i)_{i\in I}$\,.
Moreover, $\themapping{i}$ are morphisms such that:
\begin{equation}\label{commut:prop:1}
\themapping{i}=\themapping{j}\circ \themapping{i,j}\mbox{ for all }i,j\in I\mbox{ such that }i\le j\,.
\end{equation} 
\end{property}
%
%\section{Preliminary properties}
\subsubsection{Direct limit of partially defined structures}
Knowledge of the structure is not always fully available.
In some case, it may be worthwhile to build intermediate structures partially, and to infer complete structure by passing to the limit.
This section introduces a method for doing that.
\\[3pt]
In this section, it is assumed that $I=\Nset$\,, and that $(E_i,\themapping{i,j})_{\substack{i,j\in \Nset\\i\le j}}$ is a direct system, such that $E_i$ is countable for all $i\in\Nset$.
It is also defined a surjective mapping $u_i=u(i,\cdot):j\in\Nset\mapsto u(i,j)\in E_i$ for all $i\in\Nset$.
\begin{definition}[Cantor pairing function]
The Cantor pairing function is the mapping $\gamma:\Nset\times\Nset\rightarrow\Nset$ defined by $\gamma(i,j)=\frac12(i+j)(i+j+1)+i$ for all $(i,j)\in\Nset\times\Nset$\,.
\end{definition}
\begin{property}[bijection and inverse]
The Cantor pairing function is a bijection and its inverse is defined for all $n\in\Nset$ by:
\begin{equation}\label{cantorInvPairing:1}
\gamma^{-1}(n)=\bigl(c_\gamma(n),r_\gamma(n)\bigr)\;,
\mbox{ \ where }\left\{\begin{array}{l@{}}\displaystyle
w=\left\lfloor\frac{\sqrt{8n+1}-1}{2}\right\rfloor
\vspace{4pt}\\\displaystyle
r_\gamma(n)=w-n+\frac{w^2+w}2\;,
\\\displaystyle
c_\gamma(n)=n-\frac{w^2+w}2\;.
\end{array}\right.
\end{equation}
Moreover, it is noticed that $c_\gamma(n)\le n$.
\end{property}
\begin{definition}
It is defined $D_i\subset E_i\times E_i$ and $\thejointmapping{i,j}:D_i\longrightarrow D_j$ such that:
\begin{itemize}
%\item $D_0=\emptyset$\,
\item $\thejointmapping{i,j}(x_i,y_i)=(\themapping{i,j}(x_i),\themapping{i,j}(y_i))$ for all $(x_i,y_i)\in D_i$\,,
\item $D_{i+1}\supseteq \thejointmapping{i,i+1}(D_i)\cup \left\{ \left(x_{i+1},\themapping{c_\gamma(i),i+1}\circ u\circ\gamma^{-1}(i)\right) \;\Big/\; x_{i+1}\in E_{i+1}\right\}$\,.
\end{itemize}
\end{definition}
It is interesting here to explain the meaning of $\themapping{c_\gamma(i),i+1}\circ u\circ\gamma^{-1}(i)$\,.
Define the pair $(k,l)=\bigl(c_\gamma(i),r_\gamma(i)\bigr)=\gamma^{-1}(i)$\,.
Then $u\circ\gamma^{-1}(i)=y_k$ where $y_k=u_k(l)\in  E_k$\,.
In other words, $u\circ\gamma^{-1}(i)$ just does the choice of a set $E_k$ and of an element $y_k\in E_k$\,.
Then, $\themapping{c_\gamma(i),i+1}\circ u\circ\gamma^{-1}(i)=\themapping{k,i+1}(y_k)$ where $y_k=u_k(l)\in  E_k$\,.
In other words, $\themapping{c_\gamma(i),i+1}\circ u\circ\gamma^{-1}(i)$ just does the choice of a set $E_k$ and of a mapped element $\themapping{k,i+1}(y_k)\in \themapping{k,i+1}(E_k)$\,.
And it is known from the definition of the Cantor pairing function and the surjection $u$\,, that this choice will be done for all $k$ and $y_k\in E_k$\,.
This property will ensure the following lemma:
\begin{lemma}
$(D_i,\thejointmapping{i,j})_{\substack{i,j\in \Nset\\i\le j}}$ is a direct system and:
\begin{equation}\label{directlimit:prop:1}
\directlim D_i=\directlim E_i\times\directlim E_i\;.
\end{equation}
\end{lemma}
\begin{proof}By definition, $D_i \subset E_i\times E_i$ for all $i\in\Nset$.
Thus, $\directlim D_i \subset \directlim E_i\times\directlim E_i$\,. 
Let $x,y\in\directlim E_i$\,.
Then, there is $k,j\in\Nset$\,, $y_k\in E_k$ and $x_j\in E_j$ such that $\themapping{k}(y_k)=y$ and $\themapping{j}(x_j)=x$\,.
Without lost of generality, it is possible to choose $j\le k+1$\,.
Let $i=\gamma(k,l)$\,, with $l\in u_k^{-1}(y_k)$\,.
Then $i+1\ge c_\gamma(i)+1=k+1\ge j$ and $u\circ\gamma^{-1}(i)=y_k$\,.
As a consequence, $\left(\themapping{j,i+1}(x_j),\themapping{k,i+1}(y_{k})\right)\in D_{i+1}$ and $(x,y)\in\directlim D_i$.
\end{proof}
\begin{property}\label{partial:mapping:direct:limit}
Let be defined the mappings $\varphi_i:D_i\rightarrow E_i$ for all $i\in\Nset$\,.
Assume that $\themapping{i,j}\circ\varphi_{i}=\varphi_{j}\circ \thejointmapping{i,j}$ for all $i\le j$\,.
Then, there is $\varphi:\directlim E_i\times\directlim E_i\longrightarrow \directlim E_i$ such that $\varphi\circ \thejointmapping{i}=\themapping{i}\circ\varphi_{i}$ for all $i\in\Nset$\,.
\end{property}
\begin{proof}
Let $(x,y)\in\directlim E_i\times \directlim E_i$\,.
Take the smallest $s\in\Nset$ such that $(x,y)= \thejointmapping{s}(x_{ s},y_{ s})$\,,
where $(x_{ s},y_{ s})\in D_{ s}$\,.
Then, $\varphi(x,y)$ is defined by $\varphi(x,y)=\themapping{s}\circ\varphi_{ s}(x_{ s},y_{ s})$\,.
The equality $\varphi\circ \thejointmapping{i}=\themapping{i}\circ\varphi_{i}$ is then implied by the definition.
\end{proof}
\subsection{Boolean algebra}\label{sect:bool}
All along this paper, the notion of Boolean algebra is widely used and it is assumed that the reader is familiar with basic definitions and properties.
An introduction to these notions is found in~\cite{whitesitt2010} or in~\cite{birkhoff:maclane:1998}.
\subsubsection{Definition}
Let $(E,\cap,\cup,\sim,\everTrue,\everFalse)$ be a sextuple, where $E$ is a set, $\everTrue,\everFalse\in E$ and $\cap,\cup,\sim$ are respectively binary, binary and unary operators on $E$.
\begin{definition}[Boolean algebra]
$(E,\cap,\cup,\sim,\everTrue,\everFalse)$ is a Boolean algebra if:
\begin{itemize}
\item $\cap$ and $\cup$ are commutative, associative and mutually distributive,
\item (absorption) $x\cap(x\cup y)=x$
and
$x\cup(x\cap y)=x$ for all $x,y\in E$\,,
\item (complements)
$x\;\cap \sim x=\everFalse$
and
$x\;\cup \sim x=\everTrue$ for all $x\in E$\,.
\end{itemize}
\end{definition}
\begin{definition}[Boolean morphism]
Let $(E,\cap,\cup,\sim,\everTrue,\everFalse)$ and $(F,\cap,\cup,\sim,\everTrue,\everFalse)$ be Boolean algebras.
A morphism $\themapping{}: E\rightarrow F$ is a mapping from $E$ to $F$ such that
$\themapping{}(\sim x)=\sim\themapping{}(x)$
and
$\themapping{}(x\cap y)=\themapping{}(x)\cap \themapping{}(y)$
for all $x,y\in E$\,.
\end{definition}
\begin{theorem}[Stone's representation theorem~\cite{marchal:Stone}\,]
Any Boolean algebra is isomoporph to a set of subsets of a set.
\end{theorem}
\begin{definition}[generating subset]
Let
% $(E,\cap,\cup,\sim,\everTrue,\everFalse)$ be a Boolean algebra, and 
$F\subset E$\,.
Then, $F$ is called a \emph{generating subset} of the Boolean algebra $E$ if there is only one subalgebra of $E$ containing $F$
-- de facto, this subalgebra is $E$ itself.
\end{definition}
\begin{definition}[partition]
%Let $(E,\cap,\cup,\sim,\everTrue,\everFalse)$ be a Boolean algebra.
Let $F\subset E\setminus\{\emptyset\}$ be a \emph{finite set} such that $x\cap y=\everFalse$ for all $x,y\in F$ and $\bigcup_{y\in F}y=\everTrue$\,.
Then $F$ is called a \emph{partition} of $E$\,.
\end{definition}
\begin{property}[generating partition]
Let $(E,\cap,\cup,\sim,\everTrue,\everFalse)$ be a \emph{finite} Boolean algebra.
Then there is a unique partition $F\subset E$ which is a generating subset of $E$\,.
This generating partition is denoted $\thepartof{E}$\,, subsequently.
\end{property}
\begin{property}[direct limit]
Let $(I,\le)$ be a directed set.
Let $(E_i,\cap_i,\cup_i,\sim_i,\everFalse_i,\everTrue_i)$ be a Boolean algebra defined for all $i\in I$\,, and let $\themapping{i,j}:\everTrue_i\longrightarrow \everTrue_j$ be Boolean morphism for all $i\le j$.
Then, $(E_i,\themapping{i,j})_{\substack{i,j\in I\\i\le j}}$ is a directed system and its direct limit $E=\directlim E_i$ is a Boolean algebra characterized by $\themapping{i}(x_i)\cap \themapping{i}(y_i)=\themapping{i}(x_i\cap_i y_i)$
and
$\sim \themapping{i}(x_i)=\themapping{i}(\sim_i x_i)$
for all $i\in I$, $x_i,y_i\in E_i$\,.
\end{property}
\subsubsection{Probability on Boolean algebras}\label{sect:prob:bool:1}
It is given the Boolean algebra $(E,\cap,\cup,\sim,\everFalse,\everTrue)$
and an ordered field $(\RsetAlt,+,\cdot,0,1,\le)$\,.
\begin{definition}
%Let $(\RsetAlt,+,\cdot,0,1,\le)$ be an ordered field.
%Let $(E,\cap,\cup,\sim,\everFalse,\everTrue)$ be a Boolean algebra.
A mapping $P:E\longrightarrow \RsetAlt$ is a $\RsetAlt$-probability distribution on $E$ if \ 
%:
%\begin{itemize}
%\item 
%$$
$P\ge0\;,$
$
P(\everFalse)=0\;,$
%\ \ 
$P(\everTrue)=1$
%\mbox{ \ and \ }
and
$P(x\cap y)+P(x\cup y)=P(x)+P(y)\mbox{ for all }x,y\in E\,.
$
%$$
%\end{itemize}
\end{definition}
Notice that this definition relies on a finite additivity property, which makes possible the use of any ordered field.
\begin{notation}
The set of $\RsetAlt$-probability distributions defined on a Boolean algebra $E$ is denoted $\thesetofProb{\RsetAlt}{E}$\,.
From now on, the prefix $\RsetAlt$- may be omited in the particular case $\RsetAlt=\Rset$\,.
\end{notation}
\begin{property}%[probability mass function]
Assume that $E$ is \emph{finite}.
Then
%:
%\begin{itemize}
%\item 
$%\displaystyle
\sum_{x\in\thepartof{E}}P(x)=1$ for all $P\in\thesetofProb{\RsetAlt}{E}$\,.
\\[3pt]
%\item 
If there is a mapping $p:\thepartof{E}\rightarrow \RsetAlt_+$ such that $%\displaystyle
\sum_{x\in\thepartof{E}}p(x)=1$\,, then there is a unique $\RsetAlt$-probability density $P\in\thesetofProb{\RsetAlt}{E}$ such that $P(x)=p(x)$ for all $x\in\thepartof{E}$\,.
$p$ is called the ${\RsetAlt}$-probability mass of the distribution $P$\,.
%\end{itemize}
\end{property}
\begin{definition}
Let $P\in\thesetofProb{\RsetAlt}{E}$\,.
$P$ is said to be strictly positive, \emph{i.e.} $P>0$, when $P(x)>0$ for all $x\in E\setminus\{\everFalse\}$\,.
\end{definition}
For the purpose of this paper, we introduce the notion of \emph{tangible Boolean algebra}
\begin{definition}[tangible Boolean algebra]\label{def:tangible:Boolean:algebra:1}
The Boolean algebra $E$ is ${\RsetAlt}$-tangible, if there is a $\RsetAlt$-probability distribution $P\in\thesetofProb{\RsetAlt}{E}$ such that $P>0$\,.
\end{definition}
\begin{example}
The free Boolean algebra generated by $I$ is $\RsetAlt$-tangible, since it is defined a distribution $P>0$ by setting
$P\left(\bigcap_{x\in J}x\right)=2^{-\thecardof{J}}$
for all finite set $J\subset I$\,.
\\[5pt] 
However, it is known that the power set $2^I$ is $\Rset$-tangible if and only if $I$ is countable.
%In particular, for all $P\in\thesetofProb{}{2^I}$\,, it is derived:
%$$
%\max_{\substack{J\subset I\\J\mbox{ is finite}}}\sum_{x\in J}P(x)\le 1\;,
%$$
%and as a consequence, $\left\{\left.x\in I\;\right/\;P(x)>0\right\}$ is countable.
\end{example}
These examples illustrated the fact that the notion of $\RsetAlt$-tangible Boolean algebra is not directly related to the cardinality of the algebra.
The characterization of Boolean algebras which admit strictly positive finite $\Rset$-measure is still an open question~\cite{dzamonja2008}.
However, it is possible to extends any ordered field in order to ensure that the Boolean algebra is tangible.
This is the purpose of main lemma~\ref{mainlemma:1}, which is proved and used subsequently.
This result is simple but new, as far as the author knows.
\subsubsection*{Tangible Boolean algebra and ordered field extension.}%\label{sect:prob:tang:extent:1}
From now on, notation $\thegroupAlone$ is used for an ordered abelian group.
The main lemma uses the notion of \emph{Hahn series}, which allows ordered fields extensions of arbitrary cardinality~\cite{hahnSeries:ref:1,normAlling:1987}.
\begin{definition}[Hahn series]
Being given a commutative ring $\RsetAlt$ and an ordered abelian group $\thegroupAlone$, the ring of Hahn series $%\left(
\theHahnSeriesField{\RsetAlt}{\thegroupAlone}%,+,\cdot,0,1\right)
$ consists of the formal series
$%$
f=\sum_{i\in I}a_i X^i\;,
$%$
where $I$ is a well-ordered subset of $\thegroupAlone$ and $a_I\in\RsetAlt^I$\,.
The support of $f$ is $\theHahnSupport{f}=\{i\in I/ a_i\ne0\}$\,.
For $H\subset \thegroupAlone$\,, it is also defined $\theHahnSeriesField{\RsetAlt}{H}=\left\{\left.f\in\theHahnSeriesField{\RsetAlt}{\thegroupAlone}\right/\theHahnSupport{f}\subset H\right\}$\,.
\end{definition}
\begin{property}
If $\RsetAlt$ is an ordered field, then $\theHahnSeriesField{\RsetAlt}{\thegroupAlone}$ is an ordered field, ordered by the lexicographic order $\le$ defined by:
$%$
f=\sum_{i\in I}a_i X^i\ge0
\mbox{ if and only if }
a_{\min\theHahnSupport{f}}>0\;.
$%$
%Then the structure $\bigl(\theHahnSeriesField{\RsetAlt}{\thegroupAlone},+,\cdot,0,1,\le\bigr)$ is an ordered field.
\end{property}
\begin{property}\label{pos:cone:prop:1}
Let $G_+=\{g\in G/g\ge0\}$ be the positive cone of $G$.
Then, the mapping $\projFromHahn{\RsetAlt}{\thegroup{\RsetAlt}{E}}:\theHahnSeriesField{\RsetAlt}{\thegroupAlone_+}\rightarrow \RsetAlt$\,,
defined by
$
\projFromHahn{\RsetAlt}{\thegroup{\RsetAlt}{E}}\sum_{i\in \thegroupAlone_+}a_i X^i=a_0\;,
$
is a ring morphism.
\end{property}
\begin{corollary}
Let $P\in\thesetofProb{\theHahnSeriesField{\RsetAlt}{\thegroup{E}{\RsetAlt}}}{E}$\,.
Then $\projFromHahn{\RsetAlt}{\thegroupAlone}P\in\thesetofProb{\RsetAlt}{E}$\,.
\end{corollary}
\begin{lemma}[Main lemma]\label{mainlemma:1}
Being given an ordered field $\RsetAlt$ and a Boolean algebra $E$, there is a non-trivial ordered abelian group $\thegroup{E}{\RsetAlt}$ such that $E$ is $\theHahnSeriesField{\RsetAlt}{\thegroup{E}{\RsetAlt}}$-tangible.
\end{lemma}
Proof of main lemma is done in appendix~\ref{appendix:proof:mainlemma:1}.
\begin{corollary}\label{maincor:1}
For all $P\in\thesetofProb{\RsetAlt}{E}$\,,
there is $Q\in\thesetofProb{\theHahnSeriesField{\RsetAlt}{\thegroup{E}{\RsetAlt}}}{E}$ such that $Q>0$ and $P=\projFromHahn{\RsetAlt}{\thegroup{\RsetAlt}{E}}Q$
\end{corollary}
\begin{proof}
Let $R\in\thesetofProb{\theHahnSeriesField{\RsetAlt}{\thegroup{E}{\RsetAlt}}}{E}$ such that $R>0$\,, and $g\in\thegroup{E}{\RsetAlt}$ such that $g>0$\,.
Then set $Q=(1-X^g)P+X^gR$\,.
\end{proof}
\subsection{Bayesian algebra}\label{sect:bayes:model}
\subsubsection{Partial Bayesian algebra}
In this section, a recursive construction of partial Bayesian algebras is done, starting from a finite Boolean algebra $E_0$.
Then a Bayesian algabra extending $E_0$ is deduced as the direct limit of the partial constructions.
%At this time, we will not explain about the Bayesian nature of this model;
%this task will be done in the next sections.
%
\\[4pt]
The following notations will be instrumental:
\begin{notation}[cropping]\label{definitionOfCropping:1}
Let $F$ be a subset of $E$, and $x,y\in E$\,.
The cropping of $F$ by $x$ is the set:
$$\thecropof{F}{x}=\left\{y\in F\;/\;y\subset x\right\}\,.$$
The cropping of $F$ by a pair $(x,y)$ is the product set:
$$\thecropof{F}{x,y}=\thecropof{F}{x}\times\thecropof{F}{y}\,.$$
\end{notation}
\begin{notation}
For all $(\omega,\upsilon)$\,, it is defined 
$\thetransp{}(\omega,\upsilon)=(\upsilon,\omega)$\,.
For all set of pairs $x$,
it is defined $\thetransp{}(x)=\{(\upsilon,\omega)\;/\;(\omega,\upsilon)\in x\}$ and $(\theid{}\cup\thetransp{})(x)=x\cup\thetransp{}(x)$\,.
\end{notation}
\paragraph{Initial construction.}
It is defined a \emph{finite} Boolean algebra $(E_0,\cap,\cup,\sim_0,\emptyset,\everTrue_0)$\,,
where $\everTrue_0$ is a set and $E_0$ is a set of subsets of $\everTrue_0$.
Then, $\cap,\cup,\emptyset$ are respectively the set intersection, union and the empty set, and $\sim_0$ is the set complement defined by $\sim_0 x=\everTrue_0\setminus x$ for all $x\in E_0$.
It is defined $\themapping{0,0}=\theid{E_0}$\,.
It is defined $D_0=\emptyset$\,, $\varphi_0:D_0\rightarrow E_0$ (trivially empty)\,,
$D_{0}(y)=\emptyset$ for all $y\in E_{0}$\,.
It is defined $u_0=u(0,\cdot):j\in\Nset\mapsto u(0,j)\in E_0$\,, a surjective mapping.
\paragraph{Inductive construction.}
Assume that $(E_i,\cap,\cup,\sim_i,\emptyset,\everTrue_i)$, $D_i$, $\themapping{i,j}$, $u_i=u(i,\cdot)$ and $\varphi_i:D_i\rightarrow E_i$ are constructed for all $i\le j\le n$\,.
Assume that $E_i$ is a finite set for all $i\le n$\,.
Let $b_n=\themapping{c_\gamma(n),n}\circ u\circ\gamma^{-1}(n)\in E_n$\,.
Then, it is defined by a case dependent induction:
\begin{definition}
%[case $\in\{\emptyset,\everTrue_n\}$]
\label{def:case:0}
If $b_n\in\{\emptyset,\everTrue_n\}$\,, then:
\begin{itemize}
\item Set $E_{n+1}=E_n$\,, $\everTrue_{n+1}=\everTrue_{n}$ and $\sim_{n+1}=\sim_n$\,,
\item Set $\themapping{i,n+1}=\themapping{i,n}$ for all $i\le n$, and $\themapping{n+1,n+1}=\theid{E_{n+1}}$\,,
\item Define $D_{n+1}$ and $\varphi_{n+1}$ by:
\begin{itemize}
\item $D_{n+1}=D_{n}\cup\{(x,\emptyset)/x\in D_{n+1}\}\cup\{(x,\everTrue_{n+1})/x\in D_{n+1}\}$\,,
\item $\varphi_{n+1}(x,y)=\varphi_{n}(x,y)$ for $(x,y)\in D_{n}$\,,
\item Otherwise, $\varphi_{n+1}(x,\emptyset)=\varphi_{n+1}(x,\everTrue_{n+1})=x$ for $x\in E_{n+1}$\,.
\end{itemize}
\item For all $y\in E_{n+1}$\,, define $D_{n+1}(y)=\{x\in E_{n+1}\;/\; (x,y)\in D_{n+1}\}$\,.
%It is noticed that $D_{n+1}(\sim_ny)=D_{n+1}(y)$ by construction.
\end{itemize}
\end{definition}
\begin{definition}
%[case $\not\in\{\emptyset,\everTrue_n\}$]
\label{def:case:1}
If $b_n\not\in\{\emptyset,\everTrue_n\}$\,, then:
\begin{itemize}
\item Define:
$$
\begin{array}{@{}l@{}}\displaystyle
Z_{n+1}=\Bigl\{(x,y)\in \thecropof{E_n}{\sim_n b_n}\times \thecropof{E_n}{b_n} \;\Big/\;(x,\sim_n b_n),(y,b_n)\in D_n
\vspace{5pt}\\\displaystyle
\rien\qquad\qquad\qquad\qquad\qquad\qquad\mbox{ \ and \ }x\cap\varphi_n(y,b_n)=y\cap\varphi_n(x,\sim_n b_n)=\emptyset\Bigr\}
\;,
\end{array}
$$
\item Define $\displaystyle{\preeverFalse}_{n+1}=\bigcup_{(x,y)\in Z_{n+1}
}(\theid{}\cup\thetransp{})(x\times y)$\,,
\item Define $E_{n+1}=\bigcup_{k\in\Nset}E_{n+1}^k$ and $E_{n+1}^k$ by: $$\rien\hspace{-15pt}
E_{n+1}^k=\Biggl\{\bigcup_{i=1}^k\Bigl((x_i\times y_i)\setminus {\preeverFalse}_{n+1}\Bigr)\;\Bigg/\;
\forall i,\;
(x_i,y_i)\in 
(\theid{}\cup\thetransp{})(\thecropof{E_n}{b_n}\times \thecropof{E_n}{\sim_nb_n})
\Biggr\}\;,$$
\item Define ${\everTrue}_{n+1}=
(\theid{}\cup\thetransp{})(b_n\times \sim_n b_n)
\setminus {\preeverFalse}_{n+1}
\,,$ 
and $\sim_{n+1}x={\everTrue}_{n+1}\setminus x$ for all $x\in E_{n+1}$\,,
\item Set $\themapping{i,n+1}(x)=\biggl(\Bigl(\Bigl(\themapping{i,n}(x)\cap b_n\Bigr)\times \sim_n b_n\Bigr)\cup\Bigl(\Bigl(\themapping{i,n}(x)\cap \sim_n b_n\Bigr)\times b_n\Bigr)\biggr)\setminus{\preeverFalse}_{n+1}$ for all $i\le n$ and $x\in E_i$\,; set $\themapping{n+1,n+1}=\theid{E_{n+1}}$\,,
\item Define  $D_{n+1}$ and $\varphi_{n+1}$ by:
\begin{itemize}
\item $\displaystyle D_{n+1}=\thejointmapping{n,n+1}(D_{n})\cup\bigcup_{x\in E_{n+1}}\Bigl\{(x,b_{n:n+1}),(x,\sim_{n+1}b_{n:n+1})\Bigr\}$\,,
\item $\varphi_{n+1}\Bigl(\thejointmapping{n,n+1}(x,y)\Bigr)=\themapping{n,n+1}\circ\varphi_{n}(x,y)$ for $(x,y)\in D_{n}$\,,
\item $\varphi_{n+1}(x,b_{n:n+1})=(\theid{}\cup \thetransp{})
(x\cap b_{n:n+1})
$ 
for all $x\in E_{n+1}$ such that $(x,b_{n:n+1})\not\in \thejointmapping{n,n+1}(D_{n})$\,,
\item $\varphi_{n+1}(x,\sim_{n+1} b_{n:n+1})=(\theid{}\cup \thetransp{})
(x\cap \sim_{n+1}b_{n:n+1})
$ 
for all $x\in E_{n+1}$ such that $(x,\sim_{n+1} b_{n:n+1})\not\in \thejointmapping{n,n+1}(D_{n})$\,,
\end{itemize}
where $b_{n:n+1}=\themapping{n,n+1}(b_n)=(b_n\times \sim_n b_n)\setminus\preeverFalse_{n+1}$\,.
\item For all $y\in E_{n+1}$\,, define $D_{n+1}(y)=\{x\in E_{n+1}\;/\; (x,y)\in D_{n+1}\}$\,.
%It is noticed that $D_{n+1}(\sim_ny)=D_{n+1}(y)$ by construction.
\end{itemize}
\end{definition}
It is interesting here to explain this construction, which is related to the properties of conditional probabilitilies.
%First at all, the idea is to model the conditioning $[X]Y$ by $\varphi(y,x)$ where $x,y$ are respective models of $X,Y$.
The function $(x,y)\mapsto\varphi_n(y,x)$ is a partial implementation of the conditioning $(x,y)\mapsto[x]y$\,.
%This model is of course partial at this step.
The definition~\ref{def:case:0} implements the assumptions $[\everFalse]x=[\everTrue]x= x$, which expresses the independence of any proposition $x$ with the ever-true proposition $\everTrue$ and the ever-false proposition $\everFalse$\,.
From a probabilistic point of view, these assumptions are related to the relations $P(x|\everTrue)=P(x)$ and $P(x|\everFalse)=P(x)$\,.
The first relation is an obvious consequence of the definition of conditional probabilities, but the second is a choice of a solution for equation $P(x|\emptyset)P(\emptyset)=P(x\cap\emptyset)$\,.
Actually, such a choice is related to \emph{the choice of a symmetrization of the independence relation in regards to the negation.}
\\[4pt]
In definition~\ref{def:case:1}, the elements $x\times y$ and $y\times x$, defined for $(x,y)\in \thecropof{E_n}{b_n}\times \thecropof{E_n}{\sim b_n}$\,, are implementations of $x\cap[\sim b_n]y$ and of $y\cap[b_n]x$\,.
Now, the definition of $\themapping{n,n+1}:E_n\rightarrow E_{n+1}$ is a recursive implementation of the trivial equation:
$$
x=\mbox{\small$\displaystyle(x\cap b_n)\cup(x\cap \sim b_n)=$}\bigl((x\cap b_n)\cap[\sim b_n]\sim b_n\bigr)\cup\bigl((x\cap \sim b_n)\cap[b_n]b_n\bigr)\;,
$$
which is deduced from the assumptions $[b_n]b_n=[\sim b_n]\sim b_n=\everTrue$ for $b_n\not\in\{\everFalse,\everTrue\}\,;$
from a probabilistic point of view, these assumptions are themselve related to the relations $P(b_n|b_n)=P(\sim b_n|\sim b_n)=1$\,.
The definition of $\varphi_{n+1}(x,b_{n:n+1})$ is more complex and implements the following resulting equation:
$$\begin{array}{@{}l@{}}
[b_n]\Bigl(\bigl(b_n\cap x\cap[\sim b_n] y\bigr)
\cup
\bigl(\sim b_n\cap x'\cap[b_n] y'\bigr)\Bigr)
= \mbox{\small$\displaystyle
[b_n]x\cap[b_n][\sim b_n] y
=
[b_n]x\cap[\sim b_n] y$}
\vspace{4pt}\\
\phantom{%[b_n]
\Bigl(\bigl(b_n\cap x\cap
%[\sim b_n] 
y\bigr)
\cup
\bigl(\sim b_n\cap x'\cap[b_n] y'\bigr)\Bigr)}
=
\mbox{\small$\displaystyle
\bigl(b_n\cap[b_n]x\cap[\sim b_n] y\bigr)
\cup
\bigl(\sim b_n\cap[b_n]x\cap[\sim b_n] y\bigr)
$}
\vspace{4pt}\\
\phantom{%[b_n]
\Bigl(\bigl(b_n\cap x\cap
%[\sim b_n] 
y\bigr)
\cup
\bigl(\sim b_n\cap x'\cap[b_n] y'\bigr)\Bigr)}
=
\bigl(b_n\cap x\cap[\sim b_n] y\bigr)
\cup
\bigl(\sim b_n\cap y\cap[b_n]x\bigr)
\;.
\end{array}$$
This deduction is based on some characteristic assumptions, which are that $x\mapsto[b_n]x$ is a Boolean morphism, $b_n\cap[b_n]x=b_n\cap x$ and $[b_n][\sim b_n]x=x$.
The first two assumptions come rather naturally from a probabilistic point of view.
In particular, the definition of the conditional probability, $P(b_n)P\bigl([b_n]x\bigr)=P(b_n\cap x)$\,, leads to the assumptions that $b_n\cap[b_n]x=b_n\cap x$ and that $[b_n]x$ is independent of $b_n$.
% -- as a corollary, $[b_n][b_n]x=[b_n]x$.
Then, the assumption $[b_n][\sim b_n]x=x$ is a consequence of \emph{the choice of a symmetrization of the independence relation in regards to the negation.} 
\\[4pt]
In construction of definition~\ref{def:case:1}, one has to take into account algebraic relations implied from previous constructions, and especially exclusions like $x\cap[\sim_n b_n]y=\everFalse$ or $y\cap[b_n]x=\everFalse$\,.
The set $\preeverFalse_{n+1}$ is a compilation of such exclusions, and has to be removed from the constructed implementation. 
\paragraph{Properties.}
The following properties are derived with the perspective of constructing a Bayesian extension of $E_0$ as a direct limit.
\begin{property}
$E_{n+1}$ is finite.
\end{property}
\begin{proof}
Immediate induction from the definition.
\end{proof}
\begin{property}
Let $i\in\Nset$ and $x\in E_i$\,.
Then $D_i(x)=D_i(\sim_ix)$
and:
$$
\sim_iy
\;,\ 
y\cup z
\;,\ 
\varphi_{i}(y,x)
\;,\ 
\varphi_{i}(y,\sim_ix)
\in D_i(x)\;,
$$
for all $y,z\in D_i(x)$\,.
\end{property} 
\begin{proof}
Immediate induction from the definition.
\end{proof}
\begin{property}[transitive mapping]$\themapping{j,k}\circ\themapping{i,j}=\themapping{i,k}$ for all $i,j,k\in\Nset$ such that $i\le j\le k$\,.
\end{property}
\begin{proof}
True for $k=0$.
\\
Now, assume the property for $k\le n$\,.
\\
By definition, it is clear that $\themapping{n+1,n+1}\circ\themapping{i,n+1}=\themapping{i,n+1}$\,.
\\
Now assume $j\le n$\,.
Since $\themapping{j,n}(\themapping{i,j}(x))=\themapping{i,n}(x)$\,, it follows from the definition of $\themapping{i,n+1}$ that $\themapping{j,n+1}(\themapping{i,j}(x))=\themapping{i,n+1}(x)$\,.
\end{proof}
\begin{lemma}\label{lemma:1}Let $A,B,C,D$ be sets. Then:
\begin{itemize}
\item $A\cap D=B\cap C=\emptyset$ implies $(A\cup C)\cap(B\cup D)=(A\cap B)\cup(C\cap D)$\,,
\item $A\subset B$\,, $C\subset D$ and $B\cap D=\emptyset$ imply $(B\cup D)\setminus(A\cup C)=(B\setminus A)\cup (D\setminus C)$\,.
\end{itemize}
\end{lemma}
\begin{property}[Boolean morphism]
Let $i,j\in\Nset$ such that $i\le j$\,.
Let $x,y\in E_i$\,.
Then $\themapping{i,j}(x\cap y)=\themapping{i,j}(x)\cap\themapping{i,j}(y)$\,,
$\themapping{i,j}(\everTrue_i)=\everTrue_j$ and $\themapping{i,j}(\sim_i x)=\sim_j\themapping{i,j}(y)$\,.
As a consequence, $\themapping{i,j}$ is a Boolean morphism.
\end{property}
\begin{proof}
The properties are obviously true for $j=0$\,.
\\
Assume the properties for $j\le n$\,.
Then:
$$\begin{array}{@{}l@{}}\displaystyle
\themapping{i,n+1}(x_i\cap y_i)=\themapping{n,n+1}\circ\themapping{i,n}(x_i\cap y_i)=\themapping{n,n+1}\bigl(\themapping{i,n}(x_i)\cap \themapping{i,n}(y_i)\bigr)\;,
\vspace{4pt}\\\displaystyle
\themapping{i,n+1}(\everTrue_i)=\themapping{n,n+1}\circ\themapping{i,n}(\everTrue_i)=\themapping{n,n+1}(\everTrue_n)\;,
\vspace{4pt}\\\displaystyle
\themapping{i,n+1}(\sim_i x_i)=\themapping{n,n+1}\circ\themapping{i,n}(\sim_i x_i)=\themapping{n,n+1}\bigl(\everTrue_n\setminus\themapping{i,n}(x_i)\bigr)\;,
\end{array}$$
for all $x_i, y_i\in E_i$\,.
Now, let consider the only difficult case, that is $b_n\not\in\{\everFalse,\everTrue_n\}$\,.
\\
It is recalled that $\themapping{n,n+1}(y_n)=\biggl(\Bigl(\Bigl(y_n\cap b_n\Bigr)\times \sim_n b_n\Bigr)\cup\Bigl(\Bigl(y_n\cap \sim_n b_n\Bigr)\times b_n\Bigr)\biggr)\setminus{\preeverFalse}_{n+1}$ for all $y_n\in E_n$\,.
It is first deduced $\themapping{i,n+1}(\everTrue_i)=\everTrue_{n+1}$ from the definition of $\everTrue_{n+1}$\,.
%\\
Now from $b_n\cap\sim_nb_n=\emptyset$ and lemma \ref{lemma:1}, it is deduced $\themapping{i,n+1}(\sim_i x_i)=\sim_{n+1}\themapping{i,n+1}(x_i)$ and $\themapping{i,n+1}(x_i\cap y_i)=\themapping{i,n+1}(x_i)\cap\themapping{i,n+1}(y_i)$ for all $x_i, y_i\in E_i$\,.
\end{proof}
\begin{property}[commutation]\label{prop:def:commutation}
$\themapping{i,j}\circ\varphi_i=\varphi_j\circ \thejointmapping{i,j}$ for all $i,j\in\Nset$ such that $i\le j$\,.
\end{property}
\begin{proof}An immediate consequence of the definition.\end{proof}
\begin{lemma}\label{prop:def:varphi:coherence}
Let $i\in\Nset$ such that $b_i\not\in\{\emptyset,\everTrue_i\}$\,.
Then, for all $x,y\in E_{i+1}$\,, it is proved:
\begin{enumerate}[(a)]
\item \label{prop:def:varphi:coherence:a}$\varphi_{i+1}(x,b_{i:i+1})=(\theid{}\cup \thetransp{})(x\cap b_{i:i+1})$\,,
\item \label{prop:def:varphi:coherence:b}$\varphi_{i+1}(x,\sim_{i+1} b_{i:i+1})=(\theid{}\cup \thetransp{})(x\cap \sim_{i+1} b_{i:i+1})$\,,
\item \label{prop:def:varphi:coherence:ba}$\varphi_{i+1}(b_{i:i+1},b_{i:i+1})=\everTrue_{i+1}$\,,
\item \label{prop:def:varphi:coherence:bb}$\varphi_{i+1}(\sim_{i+1}b_{i:i+1},\sim_{i+1}b_{i:i+1})=\everTrue_{n+1}$\,,
\item \label{prop:def:varphi:coherence:c}$\varphi_{i+1}(x\cap y,b_{i:i+1})=\varphi_{i+1}(x,b_{i:i+1})\cap\varphi_{i+1}(y,b_{i:i+1})$\,,
\item \label{prop:def:varphi:coherence:cb}$\varphi_{i+1}(x\cap y,\sim_{i+1}b_{i:i+1})=\varphi_{i+1}(x,\sim_{i+1}b_{i:i+1})\cap\varphi_{i+1}(y,\sim_{i+1}b_{i:i+1})$\,,
\item \label{prop:def:varphi:coherence:d}$\varphi_{i+1}(\sim_{i+1}x,b_{i:i+1})=\sim_{i+1}\varphi_{i+1}(x,b_{i:i+1})$\,,
\item \label{prop:def:varphi:coherence:db}$\varphi_{i+1}(\sim_{i+1}x,\sim_{i+1}b_{i:i+1})=\sim_{i+1}\varphi_{i+1}(x,\sim_{i+1}b_{i:i+1})$\,,
\item \label{prop:def:varphi:coherence:e}$b_{i:i+1}\cap\varphi_{i+1}(x,b_{i:i+1})=x\cap b_{i:i+1}$\,,
\item \label{prop:def:varphi:coherence:f}$\sim_{i+1}b_{i:i+1}\cap\varphi_{i+1}(x,\sim_{i+1}b_{i:i+1})=x\cap \sim_{i+1}b_{i:i+1}$\,,
\item \label{prop:def:varphi:coherence:g}$\displaystyle
\varphi_{i+1}(\varphi_{i+1}(x,b_{i:i+1}),b_{i:i+1})=\varphi_{i+1}(\varphi_{i+1}(x,b_{i:i+1}),\sim_{i+1}b_{i:i+1})=\varphi_{i+1}(x,b_{i:i+1})\;,$
\item \label{prop:def:varphi:coherence:h}$\begin{array}{l@{}}\displaystyle
\varphi_{i+1}(\varphi_{i+1}(x,\sim_{i+1}b_{i:i+1}),b_{i:i+1})=
\\\displaystyle
\rien\qquad\qquad\qquad\qquad\varphi_{i+1}(\varphi_{i+1}(x,\sim_{i+1}b_{i:i+1}),\sim_{i+1}b_{i:i+1})=\varphi_{i+1}(x,\sim_{i+1}b_{i:i+1})\;.
\end{array}$
\end{enumerate}
\end{lemma}
\begin{proof}
The proofs are done by induction on $i$\,.
The results are trivial for $i=0$\,.
\\
Now assume that the results are true for $i\le n-1$\,.
\\\emph{Preliminary remark.}
Consider the greatest $k<n$ such that $\themapping{k,n}(b_k)\in\{b_{n},\sim_nb_{n}\}$\,, \emph{if its exists}\,.
Then\,:
$$
D_n(b_n)=D_n(\sim_nb_n)
= \themapping{k+1,n}\bigl(D_{k+1}(b_{k:k+1})\bigr)=\themapping{k+1,n}\bigl(D_{k+1}(\sim_{k+1}b_{k:k+1})\bigr)\;,
$$
and from induction hypothesis, the properties (\ref{prop:def:varphi:coherence:a}) to (\ref{prop:def:varphi:coherence:h}) do hold for all $x_{k+1},y_{k+1}\in D_{k+1}(b_{k:k+1})$\,.
Then, it is proved the properties for $i=n+1$ in that order:\vspace{5pt}
\\\emph{Proof of (\ref{prop:def:varphi:coherence:a}).}
The only difficult point is for $x=\themapping{n,n+1}(x_n)$ such that $x_n\in D_n(b_n)$\,.
Then:
$$
\varphi_{n+1}(x,b_{n:n+1})=\biggl(\Bigl(\bigl(\varphi_n(x_n,b_n)\cap b_n\bigr)\times \sim_n b_n\Bigr)\cup\Bigl(\bigl(\varphi_n(x_n,b_n)\cap \sim_n b_n\bigr)\times b_n\Bigr)\biggr)\setminus{\preeverFalse}_{n+1}\;.
$$
Since $\varphi_{n}(x_n,b_n)\cap b_{n}=x\cap b_{n}$ from the preliminary remark, it is deduced:
$$
\Bigl(\bigl(\varphi_n(x_n,b_n)\cap b_n\bigr)\times \sim_n b_n\Bigr)\setminus{\preeverFalse}_{n+1}=
\bigl((x_n\cap b_n)\times \sim_n b_n\bigr)\setminus{\preeverFalse}_{n+1}=
x\cap b_{n:n+1}\;.
$$
Now, let consider the second component $\Bigl(\bigl(\varphi_n(x_n,b_n)\cap \sim_n b_n\bigr)\times b_n\Bigr)\setminus{\preeverFalse}_{n+1}$\,.
\\
From the preliminary remark, it is deduced 
$$\varphi_n(\sim_n x_n\cap b_n,b_n)\cap\varphi_n(x_n,b_n)\cap\sim_n b_n=\emptyset\;,$$
and
$$\begin{array}{@{}l@{}}\displaystyle
\sim_n x_n\cap b_n\cap\varphi_n(\varphi_n(x_n,b_n)\cap\sim_n b_n,\sim_n b_n)
=\sim_n x_n\cap b_n\cap\varphi_n(\varphi_n(x_n,b_n),\sim_n b_n)
\vspace{4pt}\\\displaystyle
\rien\qquad\qquad\qquad\qquad=\sim_n x_n\cap b_n\cap\varphi_n(x_n,b_n)=
\sim_n x_n\cap b_n\cap x_n
=\emptyset
\;.\end{array}$$
As a consequence, $(\varphi_n(x_n,b_n)\cap\sim_n b_n,\sim_n x_n\cap b_n)\in Z_{n+1}$\,.
Similarly, it is shown $(\varphi_n(\sim_n x_n,b_n)\cap\sim_n b_n,x_n\cap b_n)\in Z_{n+1}$\,.
As a consequence:
$$\begin{array}{@{}l@{}}\displaystyle
\Bigl(\bigl(\varphi_n(x_n,b_n)\cap \sim_n b_n\bigr)\times b_n\Bigr)\setminus{\preeverFalse}_{n+1}=
\Bigl(\bigl(\varphi_n(x_n,b_n)\cap \sim_n b_n\bigr)\times (x_n\cap b_n)\Bigr)\setminus{\preeverFalse}_{n+1}
\vspace{4pt}\\\displaystyle
\rien\qquad\qquad\qquad\qquad
=\bigl(\sim_n b_n\times (x_n\cap b_n)\bigr)\setminus{\preeverFalse}_{n+1}
=\thetransp{}(x\cap b_{n:n+1})\;.
\end{array}$$
Thus the result.
\\\emph{Proof of (\ref{prop:def:varphi:coherence:b}).}
Proof is similar to~(\ref{prop:def:varphi:coherence:b}).
\\\emph{Proof of (\ref{prop:def:varphi:coherence:ba}).}
From~(\ref{prop:def:varphi:coherence:a}), it is deduced: $$\varphi_{n+1}(b_{n:n+1},b_{n:n+1})=(\theid{}\cup \thetransp{})(b_{n:n+1})=(\theid{}\cup\thetransp{})(b_n\times \sim_n b_n)
\setminus {\preeverFalse}_{n+1}=\everTrue_{n+1}\;.$$
\emph{Proof of (\ref{prop:def:varphi:coherence:bb}).}
Proof is similar to~(\ref{prop:def:varphi:coherence:bb}).
\\\emph{Proof of (\ref{prop:def:varphi:coherence:c}), (\ref{prop:def:varphi:coherence:cb}), (\ref{prop:def:varphi:coherence:d}) and (\ref{prop:def:varphi:coherence:db}).}
Immediate corollaries of~(\ref{prop:def:varphi:coherence:a}) and (\ref{prop:def:varphi:coherence:b}).
\\\emph{Proof of (\ref{prop:def:varphi:coherence:e}).}
Since $\thetransp{}(b_{n:n+1})=\sim_{n+1}b_{n:n+1}$ by definition, it comes:
$$
b_{n:n+1}\cap\varphi_{n+1}(x,b_{n:n+1})=b_{n:n+1}\cap(\theid{}\cup \thetransp{})(x\cap b_{n:n+1})
=x\cap b_{n:n+1}
\,.
$$
\\\emph{Proof of (\ref{prop:def:varphi:coherence:f}).}
Proof is similar to~(\ref{prop:def:varphi:coherence:e}).
\\\emph{Proof of (\ref{prop:def:varphi:coherence:g}).}
It is first deduced:
$$\begin{array}{@{}l@{}}\displaystyle
\varphi_{n+1}\bigl(\varphi_{n+1}(x,b_{n:n+1}),b_{n:n+1}\bigr)=\varphi_{n+1}\bigl(b_{n:n+1}\cap\varphi_{n+1}(x,b_{n:n+1}),b_{n:n+1}\bigr)
\vspace{4pt}\\\displaystyle
\rien\qquad\qquad\qquad\qquad\qquad\qquad\qquad
=\varphi_{n+1}(b_{n:n+1}\cap x,b_{n:n+1})=\varphi_{n+1}(x,b_{n:n+1})\;.
\end{array}$$
Now, by applying property~(\ref{prop:def:varphi:coherence:a}), it comes:
$$\begin{array}{@{}l@{}}\displaystyle
\varphi_{n+1}(\varphi_{n+1}(x,b_{n:n+1}),\sim_{n+1}b_{n:n+1})=
(\theid{}\cup \thetransp{})\bigl((\theid{}\cup \thetransp{})(x\cap b_{n:n+1})\cap \sim_{n+1}b_{n:n+1}\bigr)
\vspace{4pt}\\\displaystyle
\rien\qquad\qquad\qquad\qquad
=
(\theid{}\cup \thetransp{})\bigl(\thetransp{}(x\cap b_{n:n+1})\bigr)
=
(\theid{}\cup \thetransp{})(x\cap b_{n:n+1})
=\varphi_{n+1}(x,b_{n:n+1})\;.
\end{array}$$
\\\emph{Proof of (\ref{prop:def:varphi:coherence:h}).}
Proof is similar to~(\ref{prop:def:varphi:coherence:g}).
\end{proof}
\begin{property}[conditional Boolean morphism]\label{conditional:morphism:property:1}
Let $i\in \Nset$\,, $x\in E_i$ and $y,z\in D_i(x)$\,. 
Then $\varphi_i(y\cup z,x)=\varphi_i(y,x)\cup \varphi_i(z,x)$ and $\varphi_i(\sim_i y,x)=\sim_i \varphi_i(y,x)$\,.
\end{property}
\begin{proof}
An immediate induction from definition~\ref{def:case:0}, property~\ref{prop:def:commutation} and lemma~\ref{prop:def:varphi:coherence}.
\end{proof}
\begin{property}[reflexive conditioning]\label{reflex:prop:1}
Let $i\in\Nset$.
Then:
\begin{itemize}
\item If $(\emptyset,\emptyset)\in D_i\,,$ then $\varphi_i(\emptyset,\emptyset)=\emptyset$\,,
\item If $(x,x)\in D_i\setminus \bigl\{(\emptyset,\emptyset)\bigr\}\,,$ then
$\varphi_i(x,x)=\everTrue_i$\,.
\end{itemize}
\end{property}
\begin{proof}
An immediate induction from definition~\ref{def:case:0}, property~\ref{prop:def:commutation} and lemma~\ref{prop:def:varphi:coherence}.
\end{proof}
\begin{property}[conditional inference]\label{conditional:inference:property:1}
Let $i\in \Nset$\,.
Let $x\in E_i$ and $y\in D_i(x)$\,. 
Then $x\cap\varphi_i(y,x)=x\cap y$\,.
\end{property}
\begin{proof}
An immediate induction from definition~\ref{def:case:0}, property~\ref{prop:def:commutation} and lemma~\ref{prop:def:varphi:coherence}.
\end{proof}
\begin{property}[independence]\label{independence:property:1}
Let $i\in \Nset$\,.
Let $x\in E_i$ and $y\in D_i(x)$\,.
Then:
$$
\varphi_i\bigl(\varphi_i(y,x),x\bigr)=\varphi_i\bigl(\varphi_i(y,x),\sim_ix\bigr)
=\varphi_i(y,x)\;.
$$
\end{property}
\begin{proof}
An immediate induction from definition~\ref{def:case:0}, property~\ref{prop:def:commutation} and lemma~\ref{prop:def:varphi:coherence}.
\end{proof}
\begin{lemma}\label{prop:coherence:Z:0:0:1}
Let  $i\in\Nset$\,, $x\in E_i\setminus\{\everFalse\}$ and $y,z\in D_i(x)$\,.
Then:
$$
x\cap y\cap \varphi_i(z,\sim_i x)=\emptyset
\mbox{ \ implies \ }
\sim_i x\cap z\cap \varphi_i(y,x)=\emptyset
\;.
$$
\end{lemma}
\begin{proof}
From $x\cap y\cap \varphi_i(z,\sim_i x)=\emptyset$, it is deduced:
\begin{align*}
%\begin{array}{@{}l@{}}
&\sim_i x\cap z\cap \varphi_i(y,x)=\sim_i x\cap\varphi_i(z,\sim_ix)\cap \varphi_i(x\cap y,x)
\\
&\rien\qquad\qquad\quad
=\sim_i x\cap \varphi_i\bigl(\varphi_i(z,\sim_ix),x\bigr)\cap \varphi_i(x\cap y,x)
\\
&\rien\qquad\qquad\quad
=
\sim_i x\cap \varphi_i\bigl(x\cap y\cap \varphi_i(z,\sim_ix),x\bigr)
=
\sim_i x\cap \varphi_i\bigl(\emptyset,x\bigr)=
\emptyset\qedhere
\end{align*}
\end{proof}
\begin{lemma}\label{prop:coherence:Z:1}
Let  $i\in\Nset$ such that $b_i\not\in\{\everFalse,\everTrue_i\}$\,.
Let $x\in \thecropof{E_i}{\sim_ib_i}$ and $y\in \thecropof{E_i}{b_i}$ such that:
$$
\themapping{i,i+1}(y)\cap \varphi_{i+1}\bigl(\themapping{i,i+1}(x),\sim_{i+1}b_{i:i+1}\bigr)=\emptyset
$$
or:
$$
\themapping{i,i+1}(x)\cap \varphi_{i+1}\bigl(\themapping{i,i+1}(y),b_{i:i+1}\bigr)=\emptyset
\;.
$$
Then:
$$
x=\emptyset
\mbox{ \ or \ }
y=\emptyset
\mbox{ \ or \ }
\exists (t,u)\in Z_{i+1}\,,\; x\times y\subset t\times u\;.
$$
\end{lemma}
\begin{proof}
Denote $x'=\themapping{i,i+1}(x)$\,, $y'=\themapping{i,i+1}(y)$ and $b'=b_{i:i+1}$\,.
From lemma~\ref{prop:coherence:Z:0:0:1}, it is deduced:
$$
y'\cap \varphi_{i+1}(x',\sim_{i+1}b')=\emptyset
\mbox{ \ and \ }
x'\cap \varphi_{i+1}(y',b')=\emptyset
\;.
$$
By applying the definitions to $y'\cap \varphi_{i+1}(x',\sim_{i+1}b')=\emptyset$\,, it comes:
$$
y\times x =(y\times \sim_i b_i)\cap(b_i\times x)\subset{\preeverFalse}_{n+1}
=\bigcup_{(t,u)\in Z_{n+1}}((t\times u)\cup(u\times t))\;.
$$
At this step, it is noticed that ${\preeverFalse}_{i+1}=\emptyset$ implies $x=\emptyset$ or $y=\emptyset$\,, the third conclusion being refuted.
Now, it is assumed that $x\ne\emptyset$ and $y\ne\emptyset$\,.
It is thus deduced $y\times x \subset\bigcup_{(t,u)\in Z_{i+1}}(u\times t)\,.$
For all $\omega\in y$\,, let $J_\omega=\bigl\{(t,u)\in Z_{i+1}\;\big/\; \omega\in u\bigr\}\,.$
It comes:
$$
\{\omega\}\times x \subset \left(\bigcap_{(t,u)\in J_\omega}u\right)\times\left(\bigcup_{(t,u)\in J_\omega}t\right)\,.
$$
Define $u_\omega=\bigcap_{(t,u)\in J_\omega}u$ and $t_\omega=\bigcup_{(t,u)\in J_\omega}t$\,.
Since $\varphi_i\left(t,\sim_i b_i\right)$ is defined for all $(t,u)\in Z_{i+1}$, then $\varphi_i\left(t_\omega,\sim_i b_i\right)$ is defined and:
$$
u_\omega\cap\varphi_i\left(t_\omega,\sim_i b_i\right)=
\bigcup_{(t,u)\in J_\omega}\left(\left(\bigcap_{(t,u)\in J_\omega}u\right)\cap\varphi_i(t,\sim_i b_i)\right)\subset\bigcup_{(t,u)\in J_\omega}\bigl(u\cap\varphi_i(t,\sim_i b_i)\bigr)=\emptyset\,.
$$
As a consequence:
$$
\left(\bigcup_{\omega\in y}u_\omega\right)
\cap
\varphi_i\left(\bigcap_{\omega\in y}t_\omega,\sim_i b_i\right)
\subset
\bigcup_{\omega\in y}\left(u_\omega\cap\varphi_i\left(t_\omega,\sim_i b_i\right)\right)
=\emptyset\,.
$$
By applying lemma~\ref{prop:coherence:Z:0:0:1}, it is deduced:
\[
x\times y\subset\left(\bigcap_{\omega\in y}t_\omega\right)\times\left(\bigcup_{\omega\in y}u_\omega\right)\;,
\mbox{ \ with }
\left(\bigcap_{\omega\in y}t_\omega,\bigcup_{\omega\in y}u_\omega\right)\in Z_{i+1}\;.
\qedhere\]
\end{proof}
\begin{property}[injective mapping]\label{injective:mapping:property}
$\themapping{i,j}$ is injective for all $i,j\in\Nset$ such that $i\le j$\,.
\end{property}
\begin{proof}
It is equivalent to prove that $\themapping{i,i+1}$ is injective for all $i\in \Nset$\,.
\\[4pt]
In the case $b_i\in\{\emptyset,\everTrue_i\}$\,, then $\themapping{i,i+1}=\theid{E_i}$ by definition and is injective.
\\[4pt]
Assume $b_i\not\in\{\emptyset,\everTrue_i\}$\,.
Let $x\in E_i$ such that $\themapping{i,i+1}(x)=\emptyset$\,.
Then:
$$
\themapping{i,i+1}(x\cap b_{i})\cap \varphi_{i+1}\bigl(\sim_{i+1}b_{i:i+1},\sim_{i+1}b_{i:i+1}\bigr)=\emptyset\;.
$$
By applying lemma~\ref{prop:coherence:Z:1}\,, it is deduced the disjunction of three possible conclusions:
$$
\sim_i b_i=\emptyset
\mbox{ or }
x\cap b_i=\emptyset
\mbox{ or }
\exists (t,u)\in Z_{i+1}\,,\; \sim_ib_i\times (x\cap b_i)\subset t\times u\;.
$$
It is hypothesized that $\sim_i b_i\ne\emptyset$\,.
Then, first conclusion is refuted.
Assume third conclusion.
Then necessarily, $D_i(b_i)\ne\emptyset$\,, and then $\sim_i b_i\in D_i(b_i)$ and $\varphi_i(\sim_ib_i,\sim_ib_i)=\everTrue_i$\,.
Let $(t,u)\in Z_{i+1}$ be such that $\sim_ib_i\times (x\cap b_i)\subset t\times u$\,.
Then $\sim_ib_i\subset t$ and $u=u\cap\varphi_i(t,\sim_ib_i)=\emptyset$ by applying property~\ref{reflex:prop:1}.
At last, $x\cap b_i=\emptyset\,,$ in any cases.
Similarly, it is proved $x\cap \sim_ib_i=\emptyset$ and it is deduced $x=\emptyset$\,.
\\[3pt]
At this point, it is proved that $\themapping{i,i+1}(x)=\emptyset$ implies $x=\emptyset$\,.
Since moreover, $\themapping{i,i+1}$ is a Boolean morphism, it is injective.
\end{proof}
\begin{property}\label{partial:mapping:property}
$D_{i+1}\supseteq \thejointmapping{i,i+1}(D_i)\cup \left\{ \left(x_{i+1},\themapping{c_\gamma(i),i+1}\circ u\circ\gamma^{-1}(i)\right) \;\Big/\; x_{i+1}\in E_{i+1}\right\}$ for all $i\in\Nset\,$.
\end{property}
\begin{proof}
A direct consequence of the definition $b_i=\themapping{c_\gamma(i),i}\circ u\circ\gamma^{-1}(i)$\,.
\end{proof}
\subsubsection{Bayesian algebra and probability extension (finite case)}
It happens that $(E_i,\themapping{i:j})_{\substack{i,j\in\Nset\\
i\le j}}$ and $(D_i,\thejointmapping{i:j})_{\substack{i,j\in\Nset\\
i\le j}}$ are directed systems.
Moreover, $E_i$, $\themapping{i,j}$\,, $D_i$,  $\varphi_i:D_i\rightarrow E_i$ and $u_i=u(i,\cdot)$ 
match the hypothesis of property~\ref{partial:mapping:direct:limit}\,.
Then it is deduced the following property.
\begin{property}[algebraic extension]\label{prop:bayes:extent:1}
Define the Boolean algebra $(E,\cap,\cup,\sim,\emptyset,\everTrue)$ as the direct limit of $\bigl((E_i,\themapping{i,j})\Bigr)_{\substack{i,j\in\Nset\\
i\le j}}$ and let $\themapping{i}:E_i\longrightarrow E$ be the canonical mapping.
Then:
\begin{itemize}
\item $E$ is infinite countable,
\item $\themapping{i}$ is an injective Boolean morphism,
\item For all $(x,y)\in E\times E$\,, there is $i\in \Nset$ and $(x_i,y_i)\in D_i$ such that $(x,y)=\bigl(\themapping{i}(x_i)\themapping{i}(y_i)\bigr)$\,,
\item There is a mapping $\varphi:E\times E\rightarrow E$ defined by:
$$
\varphi\bigl(\themapping{i}(x_i)\themapping{i}(y_i)\bigr) = \themapping{i}\circ\varphi_i(x_i,y_i)\mbox{ \ for all }i\in\Nset
\mbox{ and }
(x_i,y_i)\in D_i\;,
$$
\item The mapping $\varphi$ is such that:
\begin{align}
\label{eq:bool:1}&z\mapsto \varphi(z,x)\mbox{ is a Boolean automorphism of }E\,,\\
\label{eq:necess:1}&x\subset y\mbox{ implies }\varphi(y,x)=\everTrue\mbox{ or }x=\everFalse\,,\\
\label{eq:infer:1}&x\cap\varphi(y,x)=x\cap y\,,\\
\label{eq:indep:1}&\varphi\bigl(\varphi(y,x),x\bigr)=\varphi\bigl(\varphi(y,x),\sim x\bigr)=\varphi(y,x)\,,
\end{align}
for all $x,y\in E$\,.
\end{itemize}
\end{property}
\begin{proof}
Immediate consequence of properties~\ref{partial:mapping:direct:limit}, \ref{conditional:morphism:property:1},
\ref{conditional:inference:property:1},
\ref{independence:property:1},
\ref{injective:mapping:property},
and~\ref{partial:mapping:property}
\end{proof}
\begin{notation}
From now on, it is defined
$
E_{i:}=\themapping{i}(E_i)
\mbox{ \ and \ }
b_{i:}=\themapping{i}(b_i)
\,.
$
\end{notation}
Since $E_{i:}$ is finite, it will be useful to characterize its generating partition.
\begin{property}[generating partition]\label{gen:part:prop:1}\rien
\begin{itemize}
\item If $b_{n:}\in\{\emptyset,\everTrue\}$\,, then $\thepartof{E_{n+1:}}=\thepartof{E_{n:}}$\,.
\item
If $b_{n:}\not\in\{\emptyset,\everTrue\}$\,, then $\displaystyle
\thepartof{E_{n+1:}}=\bigcup_{x\in\{b_{n:},\sim b_{n:}\}}
\left\{\left.\omega\cap\varphi(\upsilon,x)\;\right/\;
(\omega,\upsilon)\in\thecropof{\thepartof{E_{n:}}}{\sim x,x}\right\}\,.$
\end{itemize}
\end{property}
\begin{proof}
The case $b_{n:}\in\{\everFalse,\everTrue\}$ being obvious, it is assumed $b_{n:}\not\in\{\everFalse,\everTrue\}$\,.
Let $\omega_1,\omega_2\in\thecropof{\thepartof{E_{n:}}}{b_{n:}}$ and $\upsilon_1,\upsilon_2\in\thecropof{\thepartof{E_{n:}}}{\sim b_{n:}}$\,.
Then:
$$
(\omega_1,\upsilon_1)\ne(\omega_2,\upsilon_2)
\mbox{ \ implies \ }
\left\{\begin{array}{l@{}}
\bigl(\omega_1\cap \varphi(\upsilon_1,\sim b_{n:})\bigr)
\cap
\bigl(\omega_2\cap \varphi(\upsilon_2,\sim b_{n:})\bigr)=\emptyset\;,
\vspace{3pt}\\
\bigl(\upsilon_1\cap \varphi(\omega_1,b_{n:})\bigr)
\cap
\bigl(\upsilon_2\cap \varphi(\omega_2,b_{n:})\bigr)=\emptyset\;.
\end{array}\right.
$$
Now:
$$
\bigcup_{(\omega,\upsilon)\in\thecropof{\thepartof{E_{n:}}}{\sim b_{n:},b_{n:}}}
\omega\cap \varphi(\upsilon,b_{n:})= \sim b_{n:}
\mbox{ \ and \ }
\bigcup_{(\omega,\upsilon)\in\thecropof{\thepartof{E_{n:}}}{\sim b_{n:},b_{n:}}}
\upsilon\cap \varphi(\omega,\sim b_{n:})
= b_{n:}\;.
$$
Then the property is concluded as a direct consequence of lemma~\ref{prop:coherence:Z:1}\,,
which implies:
\begin{description}
\item[case 0]$b_{i:}\ne b_{n:}$\, for all $i<n$\,.
$$
\upsilon\cap \varphi(\omega,\sim b_{n:})\ne\emptyset
\iff
\omega\cap \varphi(\upsilon,b_{n:})\ne\emptyset
\iff
(\omega,\upsilon)\in
\thecropof{\thepartof{E_{n:}}}{\sim b_{n:},b_{n:}}
\;,
$$
\item[case 1] It is defined the greatest $k<n$ such that $b_{k:}= b_{n:}$\,.
$$
\upsilon\cap \varphi(\omega,\sim b_{n:})\ne\emptyset
\iff
\omega\cap \varphi(\upsilon,b_{n:})\ne\emptyset
\iff
(\omega,\upsilon)\in\!\!\!\!\!\!\!\!\!\!\bigcup_{
(y,z)\in\\\thecropof{\thepartof{E_{k+1:}}}{\sim b_{n:},b_{n:}}}
\!\!\!\!\!\!\!\!\!\!\thecropof{\thepartof{E_{n:}}}{y,z}
\,.
$$
\end{description}
But it happens that
$\displaystyle
\thecropof{\thepartof{E_{n:}}}{\sim b_{n:},b_{n:}}
=
\bigcup_{
(y,z)\in\\\thecropof{\thepartof{E_{k+1:}}}{\sim b_{n:},b_{n:}}}
\!\!\!\!\!\!\!\!\!\!\thecropof{\thepartof{E_{n:}}}{y,z}
\,.
$
\end{proof}
Property~\ref{prop:bayes:extent:1} says that it is possible to build an extension of any finite Boolean algebra by constructing a conditional operator $\varphi$ as a result of a direct limit.
Now, it is necessary to show that this conditional operator is actually an algebraic implementation of the probabilistic conditioning. 
This is done by an inductive construction of the conditional probabilities.
\\[4pt]
It is assumed now that $\RsetAlt$ is an ordered field, and $\Pi$ is a $\RsetAlt$-probability distribution defined on $E_0$\,, such that $\Pi>0$\,.
The $\RsetAlt$-probability distributions, $P_{i}\in\thesetofProb{\RsetAlt}{E_{i:}}$\,, are constructed by induction for all $i\in\Nset$\,.
\paragraph{Initial construction.}
Let $P_{0}$ be defined by $P_{0}\circ\themapping{0}=\Pi$\,.
\paragraph{Inductive construction.}
It is assumed that the $\RsetAlt$-probability distributions $P_{i}$ are defined and strictely positive for all $i\le n$\,.
The $\RsetAlt$-probability distribution $P_{n+1}$ is constructed by means of the partitions:
\begin{definition}\rien
\begin{itemize}
\item If $b_{n:}\in \{\emptyset,\everTrue\}$\,, then 
$P_{n+1}=P_{n}$\,.
\item If $b_{n:}\not\in\{\emptyset,\everTrue\}$\,, then $\displaystyle
P_{n+1}\bigl(\omega\cap\varphi(\upsilon,x)\bigr)=
\frac{P_{n}(\omega)P_{n}(\upsilon)}{P_{n}(x)}
$
for all
$x\in\{b_{n:},\sim b_{n:}\}$
and
$(\omega,\upsilon) \in \thecropof{\thepartof{E_{n:}}}{\sim x,x}\,.$
\end{itemize}
\end{definition}
\begin{property}
$P_{n+1}$ is a $\RsetAlt$-probability distribution on $E_{n+1:}$ and $P_{n+1}>0$\,.
\end{property}
\begin{proof}
$P_{n+1}>0$ by definition.
Now, it is shown that the total probability is $1$\,.
The case $b_{n:}\in\{\everFalse,\everTrue\}$ being obvious, it is assumed $b_{n:}\not\in\{\everFalse,\everTrue\}$\,.
It is first deduced for $x\ne\everFalse$\,:
$$\rien\hspace{-20pt}
\sum_{(\omega,\upsilon) \in \thecropof{\thepartof{E_{n:}}}{\sim x,x}} \frac{P_{n}(\omega)P_{n}(\upsilon)}{P_{n}(x)}
=
\sum_{\omega \in \thecropof{\thepartof{E_{n:}}}{\sim x}}
P_{n}(\omega) \frac{\displaystyle\sum_{\upsilon \in \thecropof{\thepartof{E_{n:}}}{x}}P_{n}(\upsilon)}{P_{n}(x)}
=P_{n}(\sim x)\;.
$$
As a consequence:
\[
\sum_{\substack{x\in\{b_{n:},\sim b_{n:}\}\\(\omega,\upsilon) \in \thecropof{\thepartof{E_{n:}}}{\sim x,x}}} \frac{P_{n}(\omega)P_{n}(\upsilon)}{P_{n}(x)}
=1\;.
\qedhere\]
\end{proof}
\begin{property}
$P_{i}\subset P_{j}$\,, \emph{i.e.} $P_{j}(x)=P_{i}(x)$ for all $x\in E_{i:}$ and $i,j\in\Nset$ such that $i\le j$\,.
\end{property}
\begin{proof}
It is equivalent to prove this result for $j=i+1$\,.
The case $b_{i:}\in\{\everFalse,\everTrue\}$ being obvious, it is assumed $b_{i:}\not\in\{\everFalse,\everTrue\}$\,. 
\\
Let $x\in E_{i:}$\,.
Then $x=\bigcup_{\substack{y\in\{b_{i:},\sim b_{i:}\}\\
\omega\in\thecropof{\thepartof{E_{i:}}}{x\cap y}}}\omega$\,,
and:
$$
x=
\bigcup_{\substack{y\in\{b_{i:},\sim b_{i:}\}\\
\omega\in \thecropof{\thepartof{E_{i:}}}{x\cap \sim y}}}\left(\omega\cap \bigcup_{\upsilon\in \thecropof{\thepartof{E_{i:}}}{y}}\varphi(\upsilon,y)\right)=
\bigcup_{\substack{y\in\{b_{i:},\sim b_{i:}\}\\
(\omega,\upsilon)\in \thecropof{\thepartof{E_{i:}}}{x\cap \sim y,y}}}\bigl(\omega\cap\varphi(\upsilon,y)\bigr)\;.
$$
It is thus derived:
\[
\rien\hspace{-20pt}P_{i+1}(x)=
\sum_{\substack{y\in\{b_{i:},\sim b_{i:}\}\\
(\omega,\upsilon)\in \thecropof{\thepartof{E_{i:}}}{x\cap \sim y,y}}}
\frac{P_{i}(\omega)P_{i}(\upsilon)}{P_{i}(y)}
=
\sum_{\substack{y\in\{b_{i:},\sim b_{i:}\}\\
\omega\in \thecropof{\thepartof{E_{i:}}}{x\cap \sim y}}}
P_{i}(\omega)
\frac{\displaystyle\sum_{\upsilon\in \thecropof{\thepartof{E_{i:}}}{y}}P_{i}(\upsilon)
}{P_{i}(y)}
=P_{i}(x)
\;.
\qedhere\]
\end{proof}
As an immediate consequence of the previous constructions, it is deduced:
\begin{property}[probability extension]\label{finite:proba:ext:1}
The $\RsetAlt$-probability distribution, $P\in\thesetofProb{\RsetAlt}{E}$ defined by
$P(x)=P_{i}(x)$ for all $i\in\Nset$ and $x\in E_{i:}$\,, verifies:
\begin{align}
\label{prob:proj:property:1}&P\circ\themapping{0}=\Pi\;,\\
\label{prob:prod:property:1}&
P\bigl(\sim x\cap y\cap\varphi(x\cap y,x)\bigr)=
\frac{P(\sim x\cap y)P(x\cap y)}{P(x)}\;,\mbox{ for all }y\in E\mbox{ and }x\in E\setminus\{\everFalse,\everTrue\}\,.
\end{align}
\end{property}
In regards to the following lemma, we have actually shown that $\varphi$ is an algebraic implementation of the probabilistic conditioning, at least for \emph{strictely positive} distributions.
\begin{lemma}\label{lemma:for:theorem:proof:1}
Assume $P\in\thesetofProb{\RsetAlt}{E}$ and let $x,y\in E$\,.
Then:
$$
P(x)P\bigl(\sim x\cap y\cap\varphi(x\cap y,x)\bigr)=
P(\sim x\cap y)P(x\cap y)
\mbox{ \ implies \ }
P(x)P\bigl(\varphi(y,x)\bigr)=P(x\cap y)\;.
$$
\end{lemma}
\begin{proof}
Let $x,y\in E$\,.
Since $P(\everFalse)=0$\,, it is deduced:
\begin{align*}
&P(x)P\bigl(\varphi(y,x)\bigr)=P(x)P\Bigl(\varphi\bigl(x\cap(\sim x\cup y),x\bigr)\Bigr)=P(x)\Bigl(
P\Bigl(x\cap\varphi\bigl(x\cap (\sim x\cup y),x\bigr)\Bigr)
\\&
\rien\qquad\qquad\qquad\qquad
+P\Bigl(\sim x\cap(\sim x\cup y)\cap\varphi\bigl(x\cap (\sim x\cup y),x\bigr)\Bigr)
\Bigr)
=P(x)P(x\cap y)
\\&
\rien\qquad\qquad\qquad\qquad
+
P\bigl(\sim x\cap(\sim x\cup y)\bigr)P\bigl(x\cap (\sim x\cup y)\bigr)
=P(x\cap y)\;.
\qedhere\end{align*}
\end{proof}
\subsubsection{General case}
Now, the previous result is generalized to any (possibly non finite) Boolean algebra $E_0$.
\\[4pt]
Let $\thesetofFiniteSubboolean{E_0}$ the set of \emph{finite} Boolean subalgebra of $E_0$.
It is noticed that $(\thesetofFiniteSubboolean{E_0},\subset)$ is a directed set.
By applying the previous construction, there are for each $F\in\thesetofFiniteSubboolean{E_0}$ a countable Boolean algebra $\thebayesextent{F}$\,, an injective Boolean morphism $\themapping{0,F}:F\rightarrow \thebayesextent{F}$ and an operator $\varphi_{F}:\thebayesextent{F}\times \thebayesextent{F}\rightarrow \thebayesextent{F}$\,, which verify the properties~\ref{prop:bayes:extent:1} and~\ref{finite:proba:ext:1}.
Then, it is defined, for all $F,G\in\thesetofFiniteSubboolean{E_0}$ such that $F\subset G$\,, the mappings $\themapping{F,G}:\thebayesextent{F}\rightarrow \thebayesextent{G}$ by induction:
$$\left\{\begin{array}{l@{}}\displaystyle
\themapping{F,G}\circ\themapping{0,F}=\themapping{0,G}\;,
\vspace{4pt}\\
\themapping{F,G}(x\cap y)=
\themapping{F,G}(x)\cap\themapping{F,G}(y)
\mbox{ \ and \ }
\themapping{F,G}(\sim x)=\sim \themapping{F,G}(x)\;,
\vspace{4pt}\\
\themapping{F,G}\bigl(\varphi_F(x, y)\bigr)=
\varphi_G\bigl(\themapping{F,G}(x),\themapping{F,G}(y)\bigr)\;.
\end{array}\right.$$
The mappings $\themapping{F,G}$ are Boolean morphisms,\footnote{These morphisms are also injective, thanks to property~\ref{gen:part:prop:1}.} such that $\themapping{G,H}\circ\themapping{F,G}=\themapping{F,G}$ and $\themapping{F,F}=\theid{\thebayesextent{F}}$ for all $F\le G\le H$\,.
As a consequence, $\bigl(\thebayesextent{F},\themapping{F,G}\bigr)_{\substack{F,G\in\thesetofFiniteSubboolean{E_0}\\F\subset G}}$ is a direct system.
\\[4pt]
Let $E=\directlim \thebayesextent{F}$ and $\themapping{F}:\thebayesextent{F}\rightarrow E$ be the canonical Boolean morphism defined for all $F\in\thesetofFiniteSubboolean{E_0}$.
Then $\themapping{F}\circ\themapping{0,F}\subset\themapping{G}\circ\themapping{0,G}$\,, for all $F,G\in \thesetofFiniteSubboolean{E_0}$ such that $F\subset G$\,.
Define the Boolean morphism $\themapping{0}:E_0 \rightarrow E$ by $\themapping{0}=\bigcup_{F\in \thesetofFiniteSubboolean{E_0}}\themapping{0,F}$.
Since $\themapping{0,F}$ is injective for all $F\in\thesetofFiniteSubboolean{E_0}$\,, it is deduced by construction that 
$\themapping{0}$ is injective.
\\[4pt]
By construction, it is also true that $\themapping{F}\circ\varphi_F=\themapping{G}\circ\varphi_G\circ\thejointmapping{F,G}$ for all $F,G\in \thesetofFiniteSubboolean{E_0}$ such that $F\subset G$\,.
Define the mapping $\varphi: E\times E\rightarrow E$ by $\varphi\circ\thejointmapping{F}=\themapping{F}\circ\varphi_F$ for all $F\in \thesetofFiniteSubboolean{E_0}$\,.
Then $\varphi$ inherits the characteristics of the mappings $\varphi_F$\,, \emph{i.e.} (\ref{eq:bool:1}), (\ref{eq:necess:1}), (\ref{eq:infer:1}) and (\ref{eq:indep:1}).
\\[4pt]
At last, let $\RsetAlt$ be an ordered field and let $\Pi\in\thesetofProb{\RsetAlt}{E_0}$ be \emph{strictely positive.}
By construction, there is for all $F\in \thesetofFiniteSubboolean{E_0}$ a distribution $P_{F}\in\thesetofProb{\RsetAlt}{\thebayesextent{F}}$ defined by $P_{F}\circ\themapping{0,F}=\Pi|_{F}$
and
$\displaystyle P_{F}\bigl(\sim x\cap y\cap\varphi_F(x\cap y,x)\bigr)=
\frac{P_{F}(\sim x\cap y)P_{F}(x\cap y)}{P_{F}(x)}$
for all
$y\in \thebayesextent{F}$
and
$x\in \thebayesextent{F}\setminus\{\everFalse,\everTrue\}$\,.
Since the definitions are the same modulo a morphism, it is deduced:
$$
P_{G}\circ\themapping{F,G}=P_{F} \mbox{ \ for all }F,G\in\thesetofFiniteSubboolean{E_0}\mbox{ such that }F\subset G\;.
$$
Then, it is defined $P\in\thesetofProb{\RsetAlt}{E}$ by setting $P\circ\themapping{F}=P_{F}$\,, for all $F\in \thesetofFiniteSubboolean{E_0}$\,.
This distribution $P$ inherits the characteristics of $P_{F}$\,, especially the properties~(\ref{prob:proj:property:1}) and~(\ref{prob:prod:property:1}).
\\[4pt]
Compiling all the previous construction together with lemma~\ref{lemma:for:theorem:proof:1} and the fact that $\thebayesextent{F}$ is countable for all $F\in \thesetofFiniteSubboolean{E_0}$, the following proposition is derived.
\begin{property}[extension of a Boolean algebra] \label{probability:extension:prop:3}
For all Boolean algebra $E_0$\,, there is a Boolean algebra $E$\,, an injective Boolean morphism $\themapping{}:E_0\rightarrow E$ and an operator $\varphi:E\times E\rightarrow E$ such that:
\begin{itemize}
\item $\thecardof{E}=\max\{\thecardof{E_0},\thecardof{\Nset}\}$\,,
\item $\varphi$ verifies the properties (\ref{eq:bool:1}), (\ref{eq:necess:1}), (\ref{eq:infer:1}) and (\ref{eq:indep:1}),
\item Given any ordered field $\RsetAlt$ and any strictely positive $\Pi\in\thesetofProb{\RsetAlt}{E_0}$\,, there is $P\in\thesetofProb{\RsetAlt}{E}$ such that $P\circ\themapping{}=\Pi$ and $P(x)P\bigl(\varphi(y,x)\bigr)=P(x\cap y)$ for all $x,y\in E_0$\,.
\end{itemize}
\end{property}
It is time now for spelling the main theorem.
\subsection{Main theorem}\label{sect:main:theorem}
\subsubsection{Definitions of Bayesian algebras}
\begin{definition}[Bayesian algebra]
The septuple $(E,\cap,\cup,\sim,\everFalse,\everTrue,[\;])$ is called a \emph{Bayesian algebra} if $(E,\cap,\cup,\sim,\everFalse,\everTrue)$ is a Boolean algebra and the operator $[\;]$ is such that:
\begin{description}
\item[\thepropBool:] The conditional mapping $\theCondMapping{x}:z\mapsto [x]z$ is a Boolean automorphism of $E$\,,
\item[\thepropDef:] $x\subset y$ implies $[x]y=\everTrue$ or $x=\everFalse$\,,
\item[\thepropInf:] $x\cap[x]y=x\cap y$\,,
\item[\thepropInd:] $[x][x]y=[\sim x][x]y=[x]y$\,,
\end{description}
for all $x,y\in E$\,.
\end{definition}
\begin{definition}[Bayesian morphism]
Let $(E,\cap,\cup,\sim,\everFalse,\everTrue,[\;])$ and $(F,\cap,\cup,\sim,\everFalse,\everTrue,[\;])$ be two Bayesian algebras.
A mapping $\themapping{}:E\rightarrow F$ is a Bayesian morphism if $\themapping{}$ is a Boolean morphism and $\themapping{}\bigl([x]y\bigr)=\bigl[\themapping{}(x)\bigr]\themapping{}\bigl(y)$
for all $x,y\in E$\,.
\end{definition}
\paragraph{The intuition behind.}
It is not difficult to understand the characteristic {\thepropBool}\,, which implies a Boolean behavior of the conditioned proposition; it is related to the fact that conditional probabilities are actually probabilities.
The characteristic {\thepropInf} deals with the fact that a conditional proposition is a kind of inference; for example, $P(y|x)=1$ and $P(x)=1$ imply $P(x\cap y)=1$\,.
The characteristic {\thepropDef} deals with definition of a conditional proposition; it is related to the fact that $x\subset y$ and $P(x)>0$ imply $P(y|x)>0$\,, the case $P(x)=0$ being \emph{undefined}.
The characteristic {\thepropInd} means that $[x]y$ is (logically) independent of $x$ and of $\sim x$.
It could be compared to the definition of the probabilistic conditioning, which may be rewritten $P\bigl([x]y\bigr)P(x)=P(x\cap [x]y)$\,, and to its corollary $P\bigl([x]y\bigr)P(\sim x)=P(\sim x\cap [x]y)$\,.
\begin{definition}
%Let $(E,\cap,\cup,\sim,\everFalse,\everTrue,[\;])$ be a Bayesian algebra.
A mapping $P:E\longrightarrow \RsetAlt$ is a $\RsetAlt$-probability distribution on a Bayesian algebra $E$ if $P\in\thesetofProb{\RsetAlt}{E}$ and $P(x\cap y)=P\bigl([x]y\bigr)P(x)$ for all $x,y\in E$\,.
\end{definition}
\begin{notation}
The set of $\RsetAlt$-probability distributions defined on a Bayesian algebra $E$ is denoted $\thesetofBayesProb{\RsetAlt}{E}$\,.
\end{notation}
\noindent While a Bayesian algebra~$E$ is Boolean, a $\RsetAlt$-probability distribution on the Boolean algebra~$E$ is not necessary a $\RsetAlt$-probability distribution on the \emph{Bayesian} algebra~$E$.
This explains how Lewis' triviality is avoided (by definition) by a Bayesian extension.
\begin{definition}[Bayesian extension]
Let $(B_{oole},\cap,\cup,\sim,\everFalse,\everTrue)$ be a Boolean algebra.
A Bayesian algebra $(B_{ayes},\cap,\cup,\sim,\everFalse,\everTrue,[\;])$ is a Bayesian extension of $B_{oole}$ if:
\begin{itemize}
\item There is an injective Boolean morphism $\themapping{}:B_{oole}\rightarrow B_{ayes}$\,,
\item Given any ordered field $\RsetAlt$ and any $P_{oole}\in\thesetofProb{\RsetAlt}{B_{oole}}$\,, there is $P_{ayes}\in \thesetofBayesProb{\RsetAlt}{B_{ayes}}$ such that $P_{ayes}\circ\themapping{} = P_{oole}$\,.
\end{itemize}
\end{definition}
\subsubsection{Main theorem}
\begin{theorem}[Main theorem]
\label{extension:weak:finite:th:1}
All Boolean algebras, $(B_{oole},\cap,\cup,\sim,\everFalse,\everTrue)$\,, have a Bayesian extension, $(B_{ayes},\cap,\cup,\sim,\everFalse,\everTrue,[\;])$\,, such that
$\thecardof{B_{ayes}}=\max\{\thecardof{B_{oole}},\thecardof{\Nset}\}$\,.
\end{theorem}
\begin{proof}
By applying property~\ref{probability:extension:prop:3}\,, there is a Bayesian algebra $B_{ayes}$ and an injective Boolean morphism $\themapping{}:B_{oole}\rightarrow B_{ayes}$ such that $\thecardof{B_{ayes}}=\max\{\thecardof{B_{oole}},\thecardof{\Nset}\}$ and:
\begin{equation}\label{eq:proof:main:th:1}\begin{array}{@{}l@{}}
\mbox{Given any ordered field $\RsetAlt$ and any strictly positive $\Pi\in\thesetofProb{\RsetAlt}{B_{oole}}$\,,}
\\\mbox{there is $P\in \thesetofBayesProb{\RsetAlt}{B_{ayes}}$ such that $P\circ\themapping{} = \Pi$\,.}
\end{array}\end{equation}
Now, let $\RsetAlt$ be an ordered field and let $P_{oole}\in\thesetofProb{\RsetAlt}{B_{oole}}$.
By main lemma~\ref{mainlemma:1} and corollary~\ref{maincor:1}, there is an ordered abelian group $\thegroupAlone$ and a strictely positive $\theHahnSeriesField{\RsetAlt}{\thegroupAlone}$-probability distribution $\Pi$ such that $P_{oole}=\projFromHahn{\RsetAlt}{\thegroupAlone}\Pi$\,.
By~(\ref{eq:proof:main:th:1}), there is $P\in \thesetofBayesProb{\theHahnSeriesField{\RsetAlt}{\thegroupAlone}}{B_{ayes}}$ such that $P\circ\themapping{} = \Pi$\,.
Then, $\projFromHahn{\RsetAlt}{\thegroupAlone}P\in\thesetofBayesProb{\RsetAlt}{B_{ayes}}$ and $\bigl(\projFromHahn{\RsetAlt}{\thegroupAlone}P\bigr)\circ \themapping{}=P_{oole}$ by property~\ref{pos:cone:prop:1}.
\end{proof}
The introductive theorem~\ref{main:theorem:1} is an instance of main theorem in the case $\RsetAlt=\Rset$\,.
Next section applies the notion of Bayesian algebra to the domain of logic.
\section{Toward a Deterministic Bayesian Logic}
\label{sect:logic}
It is known that Boolean algebra are the models for classical propositional logic.
Similarly, the Bayesian algebras have a logical interpretation.
In this section, the \emph{Deterministic Bayesian Logic} (DBL) is introduced concisely as a logical abstraction of Bayesian algebras.
Since Bayesian algebras are Boolean algebras, this logic also implements classical logical operators.
But in addition, DBL implements a Bayesian operator, while being bivalent.
This logic has been introduced in previous works~\cite{dmbl2007}.

\subsection{Language of Deterministic Bayesian Logic}
It is defined $\mathcal{P}$\,, a set of atomic propositions.
\begin{definition}
The set $\thesetofCondProp$ of conditional propositions is defined inductively by:
\begin{enumerate}
\item \label{def:logic:step0} $\bot\in\thesetofCondProp$ and $\mathcal{P}\subset\thesetofCondProp$\,,
\item \label{def:logic:step1} $X\rightarrow Y \in\thesetofCondProp$ for all $X,Y\in\thesetofCondProp$\,,
\item $[X] Y \in\thesetofCondProp$ for all $X,Y\in\thesetofCondProp$\,.
\end{enumerate}
\end{definition}
The set of classical propositions, $\thesetofProp\subset \thesetofCondProp$\,, is defined inductively by step~\ref{def:logic:step0} and~\ref{def:logic:step1}.
\begin{definition}
The set $\thesetofDBLProp$ of Bayesian propositions is defined by:
\[
\lfloor X_1|\cdots|X_n\rfloor \in\thesetofDBLProp\mbox{ for all }X_{1:n}\in\thesetofCondProp\;.
\]
\end{definition}
\begin{notation}
Are defined 
\mbox{$\neg X\stackrel{\Delta}{=}X\rightarrow\bot\,,$}
\mbox{$X\vee Y\stackrel{\Delta}{=}\neg X\rightarrow Y\,,$}
\mbox{$X\wedge Y\stackrel{\Delta}{=}\neg (\neg X\vee \neg Y)\,,$}
\mbox{$\top\stackrel{\Delta}{=}\neg\bot$}
and
\mbox{$X\leftrightarrow Y\stackrel{\Delta}{=} (X\rightarrow Y)\wedge (Y\rightarrow X)\,.$}
The Greek uppercase letters $\Gamma,\Delta,\Lambda,\Pi$ are notations for sequences of propositions like $X_1|\cdots|X_n$ (without $\lfloor\ \rfloor$).
\end{notation}
\paragraph{Explanation of language format.}
$\bot$ and $\top$ are respectively the ever-false and ever-true propositions.
$\rightarrow$, $\neg$, $\vee$ and $\wedge$ are respectively the \emph{classical inference}, the \emph{negation}, the \emph{disjunction} and the \emph{conjunction}. 
$[\;]$ is the \emph{conditional modality.}
The delimiters $\lfloor\ |\ \rfloor$ are interpreted as meta-disjunc\-tions:
\begin{quote}
A proposition of the form $\lfloor X \rfloor$ is interpreted as \emph{($X$ is ever-true)},
while $\lfloor X | Y | Z \rfloor$ is interpreted as \emph{($X$ is ever-true) OR ($Y$ is ever-true) OR ($Z$ is ever-true)}\,. 
\end{quote}
\subsection{Semantic of Deterministic Bayesian Logic}
\newtheorem{deduction}{Deduction}
We start the construction of DBL by first defining its semantic, that is by defining an evaluation function of the truthness of the Bayesian propositions.
\begin{definition}[conditional valuation]
Let $(E,\cup,\cap,\sim,\everFalse,\everTrue,[\;])$ be a Bayesian algebra.
A valuation of $\thesetofCondProp$ in $E$ is a mapping $H_{E}:\thesetofCondProp\rightarrow E$ verifying for all $X,Y\in\thesetofCondProp$\,:
\begin{itemize}
\item $H_{E}(\bot)=\everFalse
\mbox{ \ and \ }
%\;,\ \ 
H_{E}(X\rightarrow Y)=\bigl(E\setminus H_{E}(X)\bigr)\cup H_{E}(Y)$\,,
%\mbox{ \ et \ }
\item $H_{E}([X]Y)=\bigl[H_{E}(X)\bigr]H_{E}(Y)\,.$
%$$
\end{itemize}
\end{definition}
From these valuation, it is possible to characterize the semantic of DBL.
\begin{definition}[validity in DBL]
Subsequently, $E$ is a Bayesian algebra.
\begin{itemize}
\item A proposition $\lfloor X_1|\cdots|X_n\rfloor\in\thesetofDBLProp$ is valid according to valuation $H_{E}$, denoted $H_{E}\vDash \lfloor X_1|\cdots|X_n\rfloor$\,, if there is $i\in\{1:n\}$ such that $H_{E}(X_i)=E$\,,
\item A proposition $\lfloor \Gamma\rfloor\in\thesetofDBLProp$ is valid in $E$, denoted $E\vDash \lfloor \Gamma\rfloor$\,, if $H_{E}\vDash \lfloor \Gamma\rfloor$ for all valuation $H_{E}$\,,
\item A proposition $\lfloor \Gamma\rfloor\in\thesetofDBLProp$ is valid in DBL, denoted $\vDash \lfloor \Gamma\rfloor$\,, if $E\vDash \lfloor \Gamma\rfloor$ for all Bayesian algebra $E$\,.
\end{itemize}
\end{definition}
\begin{definition}[semantic independence]\label{hdr:indep:semantique:1}
$Y$ is semantically independent (\emph{or} free) of $X$ if $\vDash \lfloor[X]Y\leftrightarrow Y\rfloor$\,.
\end{definition}
The following deductions are almost immediate from the definition of Bayesian algebras.
\begin{deduction}\label{hdr:prop:sys:axiom}Axiomatic rewriting of the definition of Bayesian algebras:
\begin{align*}
{\theAxInfIIcond}:&\ 
\mbox{If \ }\vDash \lfloor\Gamma|X\rightarrow Y\rfloor\;, \mbox{ \ then \ }\vDash \bigl\lfloor \Gamma \big| \neg X \big| [X]Y \bigr\rfloor\;,\\
{\theAxK}:&\ \vDash \bigl\lfloor [X](Y\rightarrow Z)\rightarrow \bigl([X]Y\rightarrow [X]Z\bigr)\bigr\rfloor\;,\\
{\theAxCondIIinf}:&\ \vDash \bigl\lfloor [X]Y\rightarrow (X\rightarrow Y)\bigr\rfloor\;,\\
{\theAxNeg}:&\ \vDash \bigl\lfloor [X]\neg Y\leftrightarrow \neg [X]Y\bigr\rfloor\;,\\
{\theAxInd}:&\ \mbox{If \ }\vDash \lfloor \Gamma|Y\leftrightarrow \neg X\rfloor\mbox{ \ and \ }\vDash \lfloor \Gamma|[X]Z\leftrightarrow Z\rfloor\;, \mbox{ \ then \ }\vDash \bigl\lfloor \Gamma\big|[Y]Z\leftrightarrow Z\bigr\rfloor\;.
\end{align*}
\end{deduction}
\begin{deduction}\label{hdr:prop:fullempty:univ}
All propositions are (semantically) independent of the ever-true and the ever-false propositions:
If $\vDash \lfloor\Gamma|X|\neg X\rfloor$\,, then $\vDash \lfloor\Gamma|[X]Y\leftrightarrow Y\rfloor$\,.
In particular, $\vDash \lfloor[\top]Y\leftrightarrow Y\rfloor$ et $\vDash \lfloor[\bot]Y\leftrightarrow Y\rfloor$\,.
\end{deduction}
\begin{deduction}\label{hdr:prop:cont:taut}
The ever-true and the ever-false propositions are independent of all propositions:
(a) If $\vDash \lfloor\Gamma|Y\rfloor$\,, then $\vDash \lfloor\Gamma|[X]Y\rfloor$\,.
(b) If $\vDash \lfloor\Gamma|\neg Y\rfloor$\,, then $\vDash \lfloor\Gamma|\neg[X]Y\rfloor$\,.
In particular, $\vDash \lfloor\Gamma|[X]\top\rfloor$ and $\vDash \lfloor\Gamma|\neg[X]\bot\rfloor$\,.
\end{deduction}
\begin{deduction}\label{hdr:prop:independence}
The semantic independence implies factorisations of logical equations.
\begin{align}
\label{indep:sem:log:aq:1}
&\vDash \bigl\lfloor\neg X\big|[X]X\bigr\rfloor\;.
\mbox{In particular, if }\vDash \bigl\lfloor\Gamma\big|[X]X\leftrightarrow X\bigr\rfloor
\mbox{ then }
\vDash \lfloor\Gamma|\neg X|X\rfloor\;,
\\
\label{indep:sem:log:aq:1:2}
&\mbox{If }\vDash \bigl\lfloor\Gamma\big|[X]Y\leftrightarrow Y\bigr\rfloor
\mbox{ and }
\vDash \lfloor\Gamma|X\rightarrow Y\rfloor
\mbox{ then }
\vDash \lfloor\Gamma|\neg X|Y\rfloor\;,
\\
\label{indep:sem:log:aq:2}
&\mbox{If }\vDash \bigl\lfloor\Gamma\big|[X]Y\leftrightarrow Y\bigr\rfloor
\mbox{ and }
\vDash \lfloor\Gamma|X\vee Y\rfloor
\mbox{ then }
\vDash \lfloor\Gamma|X|Y\rfloor\;,
\\
\label{indep:sem:log:aq:3}
&\begin{array}{@{}l@{}}\mbox{If }
\vDash \bigl\lfloor\Gamma\big|(X\wedge Y)\rightarrow(X\wedge Z)\bigr\rfloor\;,\ 
\vDash \bigl\lfloor\Gamma|[X]Y\leftrightarrow Y\bigr\rfloor
\\
\rien\hspace{160pt}
\mbox{ and }
\vDash \bigl\lfloor\Gamma|[X]Z\leftrightarrow Z\bigr\rfloor
\mbox{ then }
\vDash \lfloor\Gamma|\neg X|Y\rightarrow Z\rfloor\;.
\end{array}
\end{align}
\end{deduction}
The proof system of DBL is based on the axioms defined in deduction~\ref{hdr:prop:sys:axiom}. 
\subsection{Axioms and rules of Deterministic Bayesian Logic}
\begin{definition}
The axioms and rules of DBL are:
\begin{description}
\item[Classical axioms and rules]\rien
\begin{description}
\item[C:] Axioms of propositional logic,
\item[MP:] modus ponens.
$\lfloor \Gamma|X\rfloor$ and $\lfloor \Delta|X\rightarrow Y\rfloor$ imply $\lfloor \Gamma|\Delta|Y\rfloor$\,,
\end{description}
\item[Meta-rules]\rien
\begin{description}
\item[\metaP:] Meta-permutation. 
Let $\sigma$ be a permutation of $\{1:n\}$.
Then, $\lfloor X_1|\cdots|X_n\rfloor$ implies $\lfloor X_{\sigma(1)}|\cdots|X_{\sigma(n)}\rfloor$\,,
\item[\metaC:] Meta-contraction. 
$\lfloor \Gamma|X|X\rfloor$ implies $\lfloor \Gamma|X\rfloor$\,,
\item[\metaW:] Meta-weakening. 
$\lfloor \Gamma\rfloor$ implies $\lfloor \Gamma|X\rfloor$\,,
\end{description}
\item[Condditional axioms and rules]\rien
\begin{description}
\item[\theAxInfIIcond:]
$\lfloor \Gamma|X\rightarrow Y\rfloor$ implies $\bigl\lfloor \Gamma \big| \neg X \big| [X]Y \bigr\rfloor$\,,
\item[\theAxK:]  
$\bigl\lfloor [X](Y\rightarrow Z)\rightarrow \bigl([X]Y\rightarrow [X]Z\bigr)\bigr\rfloor$\,,
\item[\theAxCondIIinf:]
$\bigl\lfloor [X]Y\rightarrow (X\rightarrow Y)\bigr\rfloor$\,,
\item[\theAxNeg:] 
$\bigl\lfloor [X]\neg Y\leftrightarrow \neg [X]Y\bigr\rfloor$\,,
\item[\theAxInd:] 
$\lfloor \Gamma|Y\leftrightarrow \neg X\rfloor$ and $\lfloor \Gamma|[X]Z\leftrightarrow Z\rfloor$ imply $\bigl\lfloor\Gamma\big|[Y]Z\leftrightarrow Z\bigr\rfloor$\,.
\end{description}
\end{description}
\end{definition}
\begin{definition}[provability]
A proposition $\lfloor\Gamma\rfloor\in\thesetofDBLProp$ is proved in DBL, denoted $\vdash \lfloor\Gamma\rfloor$, if it is deduced by a sequence of axioms and rules of DBL.
\end{definition}
\begin{definition}[propositional provability]
A proposition $X\in\thesetofProp$ is proved in propositional logic, denoted $\vdash_C X$, if $\lfloor X\rfloor$ is deduced from classical axioms and modus ponens only.
\end{definition}
The (crucial) equivalence between semantics and proofs are established by \emph{completeness theorems}.
The following partial results are almost immediate.
\begin{property}\label{completude:th:direct:1}
If $\vdash \lfloor\Gamma\rfloor$\,, then $\vDash \lfloor\Gamma\rfloor$\,.
\end{property}
\begin{proof}
The result is immediate, since deduction~\ref{hdr:prop:sys:axiom} implies that the axioms of DBL are valid in DBL (the classical and meta- axioms and rules are obviously valid in DBL). 
\end{proof}
\begin{property}\label{completude:th:class:convers:1}
If $X\in\thesetofProp$ and $\vDash \lfloor X\rfloor$\,, then $\vdash_C X$\,.
\end{property}
\begin{proof}
It is known that the class of Boolean algebras constitutes a complete semantic for the propositional logic~\cite{cori2002}.
Since main theorem establishes that any Boolean algebra can be extended into a Bayesian algebra, it follows that the validity of $X\in\thesetofProp$ in DBL implies that $X$ is proved in propositional logic.
\end{proof}
The converse of property~\ref{completude:th:direct:1}, which will complete the completeness theorem, needs a bit more work.
This result is established in the following section.
\subsection{Completeness theorem of DBL}\label{section:theoremsDBL}
In order to proof the completeness, some proof deductions are needed first.
%The notion of syntactic independence is used.
%\begin{definition}[syntactic independence]\label{def:logic:indep:1}
%$Y$ is syntactically independent of $X$ if $\vdash\lfloor[X]Y\leftrightarrow Y\rfloor$\,.
%\end{definition}
%
The proofs of these deductions are done in appendix~\ref{appendix:log:proof}, except for the first which is given as an example.
\begin{deduction}\label{prop:full:univ}
If $\vdash \lfloor\Gamma|X\rfloor$ then $\vdash \lfloor\Gamma|[X]Y\leftrightarrow Y\rfloor$\,.
In particular $\vdash \lfloor\Gamma|[\top]Y\leftrightarrow Y\rfloor$\,.
\end{deduction}
\begin{proof}
From {\theAxCondIIinf}, it is deduced 
$\vdash \lfloor[X]Y\rightarrow(X\rightarrow Y)\rfloor$
 and 
$\vdash \lfloor[X]\neg Y\rightarrow(X\rightarrow \neg Y)\rfloor$\,.
As a consequence, 
$\vdash\lfloor X\rightarrow([X]Y\rightarrow Y)\rfloor$
 and 
$\vdash \lfloor X\rightarrow([X]\neg Y\rightarrow \neg Y)\rfloor$\,.
By applying {\theAxNeg}, it comes $\vdash\lfloor X\rightarrow([X]Y\leftrightarrow Y)\rfloor$\,.
Then $\vdash \lfloor\Gamma|[X]Y\leftrightarrow Y\rfloor$ follows from {\bf MP}.
\end{proof}
\begin{deduction}\label{prop:empty:univ}
If $\vdash \lfloor\Gamma|\neg X\rfloor$ then $\vdash \lfloor\Gamma|[X]Y\leftrightarrow Y\rfloor$\,.
In particular $\vdash \lfloor\Gamma|[\bot]Y\leftrightarrow Y\rfloor$\,.
\end{deduction}
\begin{deduction}\label{prop:cont:taut}
If $\vdash \lfloor\Gamma|Y\rfloor$ then $\vdash \lfloor\Gamma|[X]Y\rfloor$\,.
In particular $\vdash \lfloor\Gamma|[X]\top\rfloor$\,.
If $\vdash \lfloor\Gamma|\neg Y\rfloor$ then $\vdash \lfloor\Gamma|\neg[X]Y\rfloor$\,.
In particular $\vdash \lfloor\Gamma|\neg[X]\bot\rfloor$\,.
\end{deduction}
\begin{deduction}\label{prop:subuniv:classic}
$\vdash \lfloor[X](Y\rightarrow Z)\leftrightarrow\bigl([X]Y\rightarrow [X]Z\bigr)\rfloor$ and $\vdash \lfloor[X]\neg Y\leftrightarrow \neg[X]Y\rfloor$\,.
\end{deduction}
\begin{deduction}\label{cor:subuniv:classic}
$\vdash \lfloor[X](Y\wedge Z)\leftrightarrow\bigl([X]Y\wedge [X]Z\bigr)\rfloor$\,, $\vdash \lfloor[X](Y\vee Z)\leftrightarrow\bigl([X]Y\vee [X]Z\bigr)\rfloor$  and $\vdash \lfloor[X](Y\leftrightarrow Z)\leftrightarrow\bigl([X]Y\leftrightarrow [X]Z\bigr)\rfloor$ \,.
\end{deduction}
\begin{deduction}\label{cor:right:equiv}
If $\vdash \lfloor\Gamma|Y\leftrightarrow Z\rfloor$
 then $\vdash \lfloor\Gamma|[X]Y\leftrightarrow [X]Z\rfloor$\,.
\end{deduction}
\begin{deduction}\label{prop:cond:inf}
$\vdash \bigl\lfloor\bigl(X\wedge[X]Y\bigr)\leftrightarrow (X\wedge Y)\bigr\rfloor$\,. 
\end{deduction}
\begin{deduction}\label{prop:introspection}
$\vdash \bigl\lfloor\neg X\big|[X]X\bigr\rfloor$
\end{deduction}
\begin{deduction}\label{prop:indep/hyp}
$\vdash \bigl\lfloor[X][X]Y\leftrightarrow [X]Y\bigr\rfloor$\,. 
\end{deduction}
\begin{deduction}\label{prop:equivalence}
If $\vdash \lfloor\Gamma|W\leftrightarrow X\rfloor$
and 
$\vdash \lfloor\Gamma|Y\leftrightarrow Z\rfloor$
 then $\vdash \bigl\lfloor\Gamma\big|[W]Y\leftrightarrow [X]Z\bigr\rfloor$\,.
\end{deduction}
At this step, it is possible to introduce the notion of logical equivalence.
\begin{property}[logical equivalence]\label{the:logical:equiv:prop:1}
The relation $\equiv$, defined on ${\thesetofCondProp}$ by $X\equiv Y\stackrel{\Delta}{\iff}\vdash X\leftrightarrow Y$, is an equivalence relation, called logical equivalence.
This relation is compatible with the logical operators:
\[
X\equiv Y\mbox{ and }U\equiv V\mbox{ implies }X\rightarrow U\equiv Y\rightarrow V\mbox{ and }[X]U\equiv [Y]V\;.
\]
\end{property}
\begin{proof}
Except for the conditioning $[\;]$, this result is a well known consequence of classical axioms and modus ponens.
Then, the proof is completed by means of deductions~\ref{cor:right:equiv} and~\ref{prop:equivalence}.
\end{proof}
\begin{property}[logical equivalence]\label{the:logical:equiv:class:prop:1}
The set of logical equivalence classes of the conditional propositions, ${\thesetofCondProp}/{\equiv}$, is a Bayesian algebra.
\end{property}
\begin{proof}
An immediate consequence of the axioms and rules of DBL.
\end{proof}
\begin{theorem}[completeness theorem]\label{completude:th:1}
$\vdash \lfloor\Gamma\rfloor$ if and only if $\vDash \lfloor\Gamma\rfloor$\,.
\end{theorem}
\begin{proof}
Owing to property~\ref{completude:th:direct:1}, it is sufficent to prove that $\vDash \lfloor\Gamma\rfloor$ implies $\vdash \lfloor\Gamma\rfloor$\,.
This inference is a consequence of property~\ref{the:logical:equiv:class:prop:1}.
\end{proof}
The completeness theorem establishes the link between the class of Bayesian algebra and the system of deduction of DBL.
Our short introduction to DBL is almost done.
Next section investigates briefly some properties of DBL in regards to the notion of \emph{logical independence}.
\subsection{Logical independence}\label{subsection:discussion}
In definition~\ref{hdr:indep:semantique:1}, a notion of independence has been defined within DBL by means of the conditional inference:
\[Y\mbox{ is independent of }X \stackrel{\Delta}{\iff}\vDash\lfloor[X]Y\leftrightarrow Y\rfloor \iff\vdash\lfloor[X]Y\leftrightarrow Y\rfloor\,.\]
This independence relation is of course inspired by the probabilistic independence, $P(Y|X)=P(Y)$\,.
It is noticed however that the independence relation of DBL is not necessary symmetric, unlike the probabilistic independence.
\\[3pt]
In mathematical logic, the notion of independence refers to the impossibility to infer or refute a proposition from a set of propositions.
The relation~(\ref{indep:sem:log:aq:1:2}), which is a consequence of rule~{\theAxInfIIcond} is a good illustration of the link between the independence relation of DBL and the logical independence:
$$
\vdash \bigl\lfloor[X]Y\leftrightarrow Y\bigr\rfloor
\mbox{ \ and \ }
\vdash \lfloor X\rightarrow Y\rfloor
\mbox{ \ imply \ }
\vdash \lfloor\neg X|Y\rfloor\;.
$$
This relation is interpreted as follows: \emph{if $Y$ is independent of $X$ and $X$ infers $Y$\,, then the inference is trivial -- \emph{i.e} $X$ is a contradiction or $Y$ is a tautology}.
\\[3pt]
On the other hand, the relation~(\ref{indep:sem:log:aq:1:2}) implies immediately the following:
\begin{equation}\label{equn:of:contradiction:1}
\vdash \bigl\lfloor[X]Y\leftrightarrow Y\bigr\rfloor
\mbox{ \ and \ }
\vdash \lfloor \neg(X\wedge Y)\rfloor
\mbox{ \ then \ }
\vdash \lfloor \neg X|\neg Y\rfloor\;,
\end{equation}
which is interpreted as: \emph{if $Y$ is independent of $X$ and $X$ contradicts $Y$\,, then the contradiction is trivial -- \emph{i.e} $X$ is a contradiction or $Y$ is a contradiction}.
Given these findings, it follows that the independence relation, as defined in definition~{\theAxInfIIcond}, may be considered as a logical independence.
An interesting point is that DBL makes possible the manipulation of the concept of logical independence as a relation within the logic itself.
Owing to the extensions theorems, DBL is also a link between the notions of probabilistic conditionals/independence and logical conditionals/independence.
\section{Conclusion}
\label{sect:conclusion}
In the first part of this paper, a new algebraic structure has been introduced, extending the Boolean algebra with an operator for the algebraic representation of the Bayesian inference.
It has been shown that it is possible to construct such extension for any Boolean algebra.
This construction is such that any probability defined on a Boolean algebra may be extended to this extension in compliance with the definition of the conditional probability.
As a corollary, this result complements the triviality of Lewis, by providing a positive answer to the definition of an algebraic conditional operator by means of an extension of the space of event.
\\[3pt]
In a second part of this paper, this algebraic extension has been applied to a model-based definition of a bivalent Bayesian extension of propositional logic.
It has been shown that this logic implements intrinsically a relation of \emph{logical independence}.
Various elementary properties have been derived.
\\[3pt]
This work addressed the delicate issue of Lewis' triviality by complementing it positively.
It introduced also some new questions.
One of them is the extension of such a result to measurable spaces.
The author surmises that such an extension is possible, at least by considering some additional restrictions on the spaces.
Another point is the study of the implied logic -- in particular, the difficult question of a possible calculus system -- and its connexion with other logical systems. 
%
%********************************
% Bibliographie
%********************************

%
%
\appendix
\small
\section{Proof of main lemma}
\label{appendix:proof:mainlemma:1}
The main lemma~\ref{mainlemma:1} is equivalently rewritten according to Stone representation theorem~\cite{marchal:Stone}:
\begin{lemma}\label{mainleamma:rewritten:1}
Given an ordered field $\RsetAlt$, a set $\everTrue$ and a Boolean sub-algebra $E\subset 2^\everTrue$,
% of subsets of $\everTrue$,
there is a non-trivial ordered abelian group $\thegroupAlone$ and a probability $P\in\thesetofProb{\RsetAlt}{\theHahnSeriesField{\RsetAlt}{\thegroupAlone}}$ such that $P>0$\,.
\end{lemma}
\begin{proof}
First, it is noticed that there is a non-trivial ordered abelian group $\thegroupAlone$ such that $\thecardof{\thegroupAlone}\ge \thecardof{\everTrue}$\,.
For example, this result may be proved by:
\begin{itemize}
\item Defining a well-ordering on $\everTrue$ (typically inherited from the order on ordinals), 
\item Considering the free abelian group $\Zset^\everTrue$ together with the implied lexicographic order. 
\end{itemize}
Then, let $\sigma: \everTrue \rightarrow \thegroupAlone_+$ be an injective mapping from $\everTrue$ to the set, $\thegroupAlone_+$, of non negative elements of $G$.
Then $P$ is defined by
$
P(y)=\left(\sum_{\omega\in \everTrue}X^{\sigma(\omega)}\right)^{-1}\sum_{\omega\in y}X^{\sigma(\omega)}
\mbox{ \ for all \ }
y\in E\,.
$
\end{proof}
\section{Proofs of deductions}\label{appendix:log:proof}
For the sake of simplicity, the rules {\metaP}, {\metaC} and {\metaW} are implicitly used in the subsequent proofs.
\begin{proof}[Proof of deduction~\ref{prop:empty:univ}]
Assume $\vdash \lfloor\Gamma|\neg X\rfloor$\,.
From deduction~\ref{prop:full:univ}, it is deduced 
$\vdash \lfloor\Gamma|[\neg X]Y\leftrightarrow Y\rfloor$\,.
From classical theorem $\vdash \lfloor X\leftrightarrow \neg\neg X\rfloor$
and {\theAxInd}, it is then deduced $\vdash \lfloor\Gamma|[X]Y\leftrightarrow Y\rfloor$\,.
\end{proof}
\begin{proof}[Proof of deduction~\ref{prop:cont:taut}]
Assume $\vdash \lfloor\Gamma|Y\rfloor$\,. 
Then $\vdash \lfloor\Gamma|[X]Y\rfloor$\,, it is deduced $\vdash \lfloor\Gamma|Y\rfloor$ then $\vdash \lfloor\Gamma|X\rightarrow Y\rfloor$\,.
By applying {\theAxInfIIcond}\,, it comes $\vdash \lfloor\Gamma|\neg X|[X]Y\rfloor$\,.
From deduction~\ref{prop:empty:univ}\,, it is deduced $\vdash \lfloor\Gamma|[X]Y\leftrightarrow Y|[X]Y\rfloor$\,.
By applying MP with $\vdash \lfloor\Gamma|Y\rfloor$\,, it is deduced
$\vdash \lfloor\Gamma|[X]Y|[X]Y\rfloor$ and thus $\vdash \lfloor\Gamma|[X]Y\rfloor$\,.
\\[3pt]
Now assume $\vdash \lfloor\Gamma|\neg Y\rfloor$\,.
It is similarly deduced $\vdash \lfloor\Gamma|[X]\neg Y\rfloor$\,, and then $\vdash \lfloor\Gamma|\neg [X]Y\rfloor$ by {\theAxNeg}+MP.
\end{proof}
\begin{proof}[Proof of deduction~\ref{prop:subuniv:classic}]
This deduction is almost obtained from~{\theAxK} and~{\theAxNeg}.
However, it is necessary to prove $\vdash \lfloor\bigl([X]Y\rightarrow [X]Z\bigr)\rightarrow [X](Y\rightarrow Z)\rfloor$\,.
From $\vdash \lfloor\neg Y\rightarrow (Y\rightarrow Z)\rfloor$ and $\vdash \lfloor Z\rightarrow (Y\rightarrow Z)\rfloor$\,, it is deduced $\vdash \lfloor[X]\bigl(\neg Y\rightarrow (Y\rightarrow Z)\bigr)\rfloor$ and $\vdash \lfloor [X]\bigl(Z\rightarrow (Y\rightarrow Z)\bigr)\rfloor$\,, by applying deduction~\ref{prop:cont:taut}.
By applying {\theAxK}+MP, it comes $\vdash \lfloor[X]\neg Y\rightarrow [X](Y\rightarrow Z)\rfloor$ and $\vdash \lfloor [X]Z\rightarrow [X](Y\rightarrow Z)\bigr)\rfloor$\,.
By applying {\theAxNeg}+MP, it is also deduced $\vdash \lfloor\neg [X]Y\rightarrow [X](Y\rightarrow Z)\rfloor$\,.
As a consequence, $\vdash \lfloor\bigl([X]Y\rightarrow [X]Z\bigr)\rightarrow [X](Y\rightarrow Z)\rfloor$\,.
\end{proof}
\begin{proof}[Proof of deduction~\ref{cor:subuniv:classic}]
Immediate corollary of deduction~\ref{prop:subuniv:classic}.
\end{proof}
\begin{proof}[Proof of deduction~\ref{cor:right:equiv}]
From $\vdash \lfloor\Gamma|Y\leftrightarrow Z\rfloor$\,, it is deduced $\vdash \lfloor\Gamma|[X](Y\leftrightarrow Z)\rfloor$ by deduction~\ref{prop:cont:taut}.
Then it is deduced $\vdash \lfloor\Gamma|[X]Y\leftrightarrow [X]Z\rfloor$ from deduction~\ref{cor:subuniv:classic}.
\end{proof}
\begin{proof}[Proof of deduction~\ref{prop:cond:inf}]
From {\theAxCondIIinf}, it is deduced $\vdash \lfloor[X]\neg Y\rightarrow(X\rightarrow \neg Y)\rfloor$\,.
Then it is deduced $\vdash \lfloor\neg(X\rightarrow \neg Y)\rightarrow\neg[X]\neg Y\rfloor$\,.
Then $\vdash \lfloor(X\wedge Y)\rightarrow[X]Y\rfloor$ by applying {\theAxNeg}\,, and finally $\vdash \bigl\lfloor(X\wedge Y)\rightarrow\bigl(X\wedge [X]Y\bigr)\bigr\rfloor$ by applying {\theAxNeg}\,.
Conversely, $\vdash \lfloor[X]Y\rightarrow (X\rightarrow Y)\rfloor$ is deduced from  {\theAxCondIIinf},
and then $\vdash\bigl\lfloor \bigl(X\wedge[X]Y\bigr)\rightarrow (X\wedge Y)\bigr\rfloor$\,.
\end{proof}
\begin{proof}[Proof of deduction~\ref{prop:introspection}]
Immediate consequence of {\theAxInfIIcond}.
\end{proof}

\begin{proof}[Proof of deduction~\ref{prop:indep/hyp}]
By deductions~\ref{prop:cond:inf} and~\ref{cor:right:equiv}, it is deduced
$\vdash \bigl\lfloor[X]\bigl(X\wedge[X]Y\bigr)\leftrightarrow [X](X\wedge Y)\bigr\rfloor$\,. 
By applying deduction~\ref{cor:subuniv:classic}, it comes 
$\vdash \bigl\lfloor\bigl([X]X\wedge[X][X]Y\bigr)\leftrightarrow \bigl([X]X\wedge [X]Y\bigr)\bigr\rfloor$\,.
As a consequence, $\vdash \bigl\lfloor[X]X\rightarrow\bigl([X][X]Y\leftrightarrow [X]Y\bigr)\bigr\rfloor$\,.
Then, by applying deduction~\ref{prop:introspection} and MP, it comes
$\vdash \bigl\lfloor\neg X\bigl|[X][X]Y\leftrightarrow [X]Y\bigr\rfloor$\,. 
Now by applying deduction~\ref{prop:empty:univ}, it is deduced $\vdash \bigl\lfloor[X][X]Y\leftrightarrow [X]Y\bigl|[X][X]Y\leftrightarrow [X]Y\bigr\rfloor$ and the result follows from {\metaW}. 
\end{proof}
\begin{proof}[Proof of deduction~\ref{prop:equivalence}]
It is sufficient to prove $\vdash \lfloor\Gamma|W\leftrightarrow X\rfloor\Rightarrow\; \vdash \bigl\lfloor\Gamma\big|[W]Y\leftrightarrow [X]Y\bigr\rfloor$\,.\\
Assume $\vdash \lfloor\Gamma|W\leftrightarrow X\rfloor$\,.
Then $\vdash \lfloor\Gamma|\neg W\leftrightarrow \neg X\rfloor$\,.
Now, $\vdash \bigl\lfloor[X][X]Y\leftrightarrow [X]Y\bigr\rfloor$ by deduction~\ref{prop:indep/hyp}.
Applying {\theAxInd}, it is deduced $\vdash \bigl\lfloor\Gamma|[\neg W][X]Y\leftrightarrow [X]Y\bigr\rfloor$\,.
Now  $\vdash \lfloor\Gamma|W\leftrightarrow \neg\neg W\rfloor$\,.
Applying {\theAxInd} again, it is deduced $\vdash \bigl\lfloor\Gamma|[W][X]Y\leftrightarrow [X]Y\bigr\rfloor$\,.
\\[3pt]
Now, deduction~\ref{prop:introspection} implies $\vdash \bigl\lfloor\neg W\big|[W]W\bigr\rfloor$\,.
It is thus deduced $\vdash \bigl\lfloor\Gamma\big|\neg W\big|\bigl([W]W\wedge [W][X]Y\bigr)\leftrightarrow [X]Y\bigr\rfloor$\,.
Since $\vdash \bigl\lfloor \bigl(W\wedge [X]Y\bigr)\leftrightarrow (W\wedge Y)\bigr\rfloor$\,, it is deduced $\vdash \bigl\lfloor\Gamma\big|\neg W\big|[W](W\wedge Y)\leftrightarrow [X]Y\bigr\rfloor$
and finally $\vdash \bigl\lfloor\Gamma\big|\neg W\big|[W]Y\leftrightarrow [X]Y\bigr\rfloor$\,.
The result is concluded by applying deduction~\ref{prop:empty:univ}.
\end{proof}
\end{document}